\documentclass[a4paper, 11pt]{article}
\usepackage[T1]{fontenc}
\usepackage[utf8]{inputenc}
\usepackage[all]{xy}
\usepackage{amsfonts}
\usepackage{graphicx}

\usepackage{cite}
\usepackage{enumerate}		
\usepackage{hyperref}		
\usepackage{relsize}
\usepackage{amssymb} 		
\usepackage{mathrsfs} 		
\usepackage{bbm} 		
\usepackage{extarrows} 		
\usepackage{stmaryrd} 		
\usepackage{mathtools}		

\usepackage{amsmath}	
\usepackage{amsthm} 		
\usepackage[nottoc]{tocbibind} 

\theoremstyle{plain}
\newtheorem{thm}{Theorem}
\numberwithin{thm}{section}
\newtheorem{prop}[thm]{Proposition}

\theoremstyle{definition}
\newtheorem{defn}{Definition}
\numberwithin{defn}{section}

\numberwithin{ex}{section}

\numberwithin{nota}{section}

\newtheorem{remark}{Remark}
\newtheorem{remarks}[remark]{Remarks}
\numberwithin{remark}{section}

\usepackage{todonotes}
\setuptodonotes{color=blue!10!white}
\usepackage{tikz}
\usepackage{tikz-cd} 
\usetikzlibrary{positioning,intersections} 
\usetikzlibrary{decorations.pathmorphing} 
\usetikzlibrary{arrows}

\newcommand{\yo}{\text{\usefont{U}{min}{m}{n}\symbol{'110}}}
\DeclareFontFamily{U}{min}{}
\DeclareFontShape{U}{min}{m}{n}{<-> dmjhira}{}

\usepackage{xspace} 
 
\newcommand{\ie}{i.e.\@\xspace}

\newcommand{\ac}{`}

\renewcommand{\epsilon}{\varepsilon}
\renewcommand{\phi}{\varphi}

\newcommand{\op}{^{\textup{op}}}
\newcommand{\Vop}{^{\textup{V}}}

\newcommand{\can}{^{\textup{can}}}

\newcommand{\cartcov}{^{\textup{cart,cov}}}

\newcommand{\Gir}{\operatorname{\textup{Gir}}}

\newcommand{\Hom}{\operatorname{Hom}}

\renewcommand{\lim}{\operatorname{lim}}
\newcommand{\colim}{\operatorname{colim}}

\newcommand{\lan}{\mathrm{Lan}}

\newcommand{\comma}[2]			
{\mbox{$(#1\!\downarrow\!#2)$}}

\newcommand{\Sheafify}{\mathrm{a}}	
\newcommand{\canst}{\mathcal{S}}	

\newcommand{\CAT}{\mathbf{CAT}}
\newcommand{\Cat}{\mathbf{Cat}}
\newcommand{\Site}{\mathbf{Site}}

\newcommand{\Topos}{\mathbf{Topos}}

\newcommand{\Set}{\mathbf{Set}}

\newcommand{\Ind}{\mathbf{Ind}}

\newcommand{\Sh}{\mathbf{Sh}}

\newcommand{\Fib}{{\mathbf{Fib}}}

\newcommand{\dcat}{\mathbb{D}}

\newcommand{\icat}{\mathbb{I}}

\newcommand{\abicat}{\mathcal{A}}
\newcommand{\bbicat}{\mathcal{B}}
\newcommand{\cbicat}{\mathcal{C}}

\newcommand{\dbicat}{\mathcal{D}}

\newcommand{\gbicat}{\mathcal{G}}

\newcommand{\ibicat}{\mathcal{I}}

\newcommand{\Etopos}{\mathscr{E}}
\newcommand{\Ftopos}{\mathscr{F}}

\makeatletter
\newbox\xrat@below
\newbox\xrat@above
\newcommand{\xrightarrowtail}[2][]{%
	\setbox\xrat@below=\hbox{\ensuremath{\scriptstyle #1}}%
	\setbox\xrat@above=\hbox{\ensuremath{\scriptstyle #2}}%
	\pgfmathsetlengthmacro{\xrat@len}{max(\wd\xrat@below,\wd\xrat@above)+.6em}%
	\mathrel{\tikz [>->,baseline=-.75ex]
		\draw (0,0) -- node[below=-2pt] {\box\xrat@below}
		node[above=-2pt] {\box\xrat@above}
		(\xrat@len,0) ;}}
\makeatother

\tikzset{Rightarrow/.style={double equal sign distance,>={Implies},->},
	triple/.style={-,preaction={draw,Rightarrow}}}

 \newcommand{\biimp}
{\! \Leftrightarrow\!}

\begin{document}

\title{On morphisms of relative toposes}

\author{Léo Bartoli and Olivia Caramello}

\maketitle

\begin{abstract}
    We systematically investigate the functors between sites which induce morphisms of relative toposes. In particualar, we establish a relative version of Diaconescu's theorem, characterizing the relative geometric morphisms towards a relative sheaf topos in terms of a notion of flat (equivalently, filtered) functor relative to the base topos. 
\end{abstract}

\tableofcontents

\section{Introduction}

Many recent pioneering theories from different areas of mathematics, such as Scholze and Clausen’s condensed mathematics \cite{condensed}, Tao and Jamneshan’s topos-theoretic measure theory (cf. \cite{taorelative} and \cite{jamneshanstructure}) and Tomasic’s topos-theoretic difference algebra \cite{tomasicdifference}, have shown the interest of working over an arbitrary base topos. The usefulness of considering mathematical objects over a base topos different than the topos of sets notably lies in the possibility of `encapsulating' part of the complexity of the situation in the base topos, so that the given notions acquire a simpler or more natural expression with respect to it. In \cite{elephant}, some kind of relative topos theory has been developed by building on the notion of internal category. However, this notion is rather rigid and restrictive, since it does not encompass basic examples of what we would like to be internal sites, for instance, the canonical stack of a relative topos, with its canonical relative topology. In this light, stacks, and more generally fibrations, end up to be the right way to formalize relative topos theory. This approach has been developed in \cite{CaramelloZanfa}, which provides new foundations for relative topos theory in the language of stacks.
This theoretical framework allows us to reason ``internally'', or ``relatively to the base topos'', by using parameters coming from the base topos, and not breaking all the “pseudo” data that may arise in the mathematical practice. Moreover, we can think of stacks as generalized internal categories, and, looking at them as fibrations over the base site, equip them with Grothendieck topologies.  

The contents of the paper can be described as follows.
We first recall the basics of this theory, which provides the background for our results. Then we study, in a systematic way, the functors between sites which induce morphisms between relative toposes. In particular, we introduce relative morphisms of sites and relative flat functors. The difficulty in characterizing such functors lies on the fact that, at the site level, the structural morphisms are provided by comorphisms, but we wish the morphisms between two such structural comorphisms to be morphisms of sites, as in the “absolute” (\textit{i.e.} non-relative) setting. Hence, we cannot directly check, at the site level, the commutativity of the triangle induced at the topos level. We pursue three different approaches to address the question of whether or not the induced triangle at the topos level commutes (up to isomorphism):

\begin{itemize}
    \item First, by using relative local cofinality conditions arising from the computation of left Kan extensions, we answer this question in the greatest generality, namely for two comorphisms into some base site, with a continuous functor between them, not necessarily giving a commutative triangle, but only a “lax” one.
    \item  In the second approach, inspired by the classification theorem characterizing Giraud's classifying topos of a stack, we lift the domain of a flat functor to a site having all finite limits in order to have a simpler understanding of the flatness condition at that level. As in the absolute case, we have a notion of relative canonical site and a functor playing the role of the Yoneda functor. In this context, we obtain conditions which are the relative analogues for a functor to be filtering. Remarkably, in the light of the stack property of canonical relative sites, we can even get rid of the \textit{absolute} local filteredness condition.
    \item The last method exploits the fact that, in the same setting, we can turn the morphism of sites which we study into a Morita-equivalent comorphism of sites by the use of the abstract duality of \cite{denseness}, so that we only need to examine one type of functors, the comorphisms of sites, and check the commutativity of the morphisms between toposes directly at the site level: we implement this strategy by introducing a \textit{diagonal} site making the given diagram commute at the site, and investigate under which conditions we can restrict to this diagonal site without losing data at the topos level. 
\end{itemize}

These approaches eventually culminate into the main result of this paper (Theorem \ref{diaconescufibration}), which is a relative version of Diaconescu's theorem, providing an equivalence between relative geometric morphism towards a topos of sheaves on a relative site and a category of \ac \ac relatively locally filtering functors''.

\section{Relative toposes, relative sites}\label{chap:relative_sites}

In this section, we review the fundamental concepts of relative topos theory, as formulated in \cite{CaramelloZanfa} in the language of stacks and fibrations. By a \emph{relative topos} over $\Etopos$ we mean a geometric morphism $f: \Ftopos \to \Etopos$. It can be understood as `` $\Ftopos$ is a topos in the world of $\Etopos$ ''. In order to dispose of a notion of relative site as well, we have to also consider ``sites in the world of the base topos'': their underlying categories will be stacks, or, more generally, fibrations, over a site of definition for $\Etopos$. As reminded below, in order to define the associated relative toposes, the domain of these fibrations will be equipped with a minimal topology with respect to that of the base site, relativizing the case of the trivial topology giving the usual presheaves categories. Thus, a relative site on $({\cal C},J)$ will be a fibration over $\cal C$ equipped with a topology that contains Giraud's one. These relative sites $p: ({\cal G}({\mathbb D}),K) \to ({\cal C},J)$ induce relative toposes, as the Giraud topology is, by definition, the minimal one making the projection $p$ into a comorphism of sites. In the other direction, taking the \textit{canonical relative site} of a geometric morphism will allow us to reconstruct it up to isomorphism: every relative topos appears as a topos of sheaves on a relative site. Lastly, we exhibit a relative analogue of the canonical functor from an ordinary site to the associated topos of sheaves on it; this will be important for our study of morphisms of sites in the relative setting carried out in the subsequent sections of the paper.

\subsection{Relative sites}

For a $\cal C$-indexed category $\mathbb D$ together with a Grothendieck topology $J$ on $\cal C$, recall that $\Gir_J(\dcat)$ is the topos of sheaves on $({\mathcal{G}}({\mathbb D}),J_{\mathbb D})$ where $J_{\mathbb D}$ is the Giraud topology, having for covering those containing a set of cartesian arrows sent to a covering in $({\cal C},J)$ by $p_{\mathbb D}$. As observed in  \cite{CaramelloZanfa} Remark 5.4.1 (ii), the equivalence
\[
\Gir_J(\dcat) \simeq \Ind_\cbicat(\dcat\Vop, \canst_{(\cbicat,J)})
\]
of Corollary 5.4.1 shows that, given a small-generated site $({\cal C}, J)$ and a $\cal C$-indexed category ${\mathbb D}$, the relative topos $\Gir_J(\dcat)$ yields the appropriate notion of ``topos of $\Sh({\cal C}, J)$-valued presheaves on $\mathbb D$''. This motivates the following definition:

\begin{defn}[Definition 8.2.1 \cite{CaramelloZanfa}]
Let $({\cal C}, J)$ be a small-generated site. A \emph{relative site}\index{site!relative -} over $({\cal C}, J)$ is a site of the form $({\cal G}({\mathbb D}), K)$, where ${\mathbb D}$ is a $\cal C$-indexed category and $K$ is a Grothendieck topology on ${\cal G}({\mathbb D})$ containing the Giraud topology $J_{\mathbb D}$. Any relative site $({\cal G}({\mathbb D}), K)$ is endowed with its structure comorphism of sites $p_{\mathbb D}:({\cal G}({\mathbb D}), K) \to ({\cal C}, J)$.
\end{defn}

\begin{remark}
	In the interest of maximal generality, we do not require $\mathbb D$ to be a stack on $({\cal C}, J)$, nor to be the $\cal C$-indexing of an internal category in $\Sh({\cal C}, J)$. In fact, indexed categories simultaneously generalize stacks and internal categories, as not every stack is an internal category. Still, as we shall see, relative toposes can always be represented as toposes of sheaves on relative sites whose underlying indexed category is a stack (cf. Theorem \ref{thm:characterizationrelativetoposes}). 
\end{remark}

Recall that a comorphism of sites $p: ({\cal D}, K) \to ({\cal C},J)$ is a functor such that $C_p^* = a_K(- \circ p^{op})i_J$ is the inverse image part of a geometric morphism. Or, equivalently, they are such that for every covering sieve $S$ on some $p(d)$ in $\cal C$, there exists a covering sieve $S'$ on $d$ such that $p(S') \subseteq S$. 

\begin{thm}[Theorem 3.13 \cite{denseness}]
The projection of a fibration $p_{\mathbb D}: ({\cal G}({\mathbb D}),K) \to ({\cal C},J)$  is a comorphism of sites if and only if $K$ contains $J_{\mathbb D}$. Hence, any relative site $$p_{{\mathbb D}}:({\cal G}({\mathbb D}), K)\to ({\cal C}, J)$$ induces the relative topos
	\[
	C_{p_{\mathbb D}}:\Sh({\cal G}({\mathbb D}), J')\to \Sh({\cal C}, J)
	\]
    with the inverse image of the structure morphism given by:
    $$C_{p_{\mathbb D}}^* = a_K(-\circ p_{\mathbb D}^{op})i_J$$.
\end{thm}

Trivial relative sites are those such that the Grothendieck topology $K$ coincides with Giraud's topology $J_{\mathbb D}$; as in the classical setting, they yield the relative presheaf toposes. Accordingly, arbitrary relative sites yield arbitrary subtoposes of relative presheaf toposes. 

\begin{remarks}

\begin{enumerate}[(a)]
    \item There are some size issues involved in the notion of relative site. We do not require smallness hypotheses in our definition as, for technical reasons, it is convenient to be able to work also with large presentation sites. One can resolve such issues either by working with respect to a bigger Grothendieck universe, or by showing that the relevant sites under consideration are in fact small-generated (for instance, one can show that the canonical site of a geometric morphism, in the sense of Definition \ref{defrelativesiteofamorphism}, is small-generated, see Proposition \ref{prop:relativesitesmallgenerated} below). 
    \item It might seem restrictive at first sight to require the structure morphism in our definition of a relative site to be a fibration. In fact, we will consider what we call \textit{sites over} $({\cal C},J)$ as those comorphisms of sites $p: ({\cal D},K) \to ({\cal C},J)$. Whenever possible, we shall formulate our results at this higher level of generality. 
    \item Still, this makes no difference for the associated geometric morphisms, since with any site \textit{over} $({\cal C}, J)$ , we can associate a \textit{relative} one, namely the fibration of generalized elements of $p$, endowed with the topology coinduced by the canonical functor ${{\cal D}}\to (1_{\cal C}\downarrow p)$, which induces the same geometric morphism (cf. Theorem 3.24 \cite{denseness}). Or, at a more abstract level, we can also turn arbitrary sites \textit{over} $({\cal C},J)$ into \textit{relative} ones, with the help of the canonical relative site of the relative topos they induce. This construction will be a key tool in our study of morphisms of sites in a relative setting.
\end{enumerate}
\end{remarks}

\subsection{The canonical relative site of a relative topos}

This subsection is about the converse result, finishing to expose that we have the same kind of correspondence between relative sites and relative toposes, as between sites and toposes in the non relative setting:

\begin{thm}\label{thm:characterizationrelativetoposes}
	Let $({\cal C}, J)$ be a small-generated site. Every relative topos $f: \Etopos\to \Sh({\cal C}, J)$ is of the form $C_{p_{\mathbb D}}$ for some relative site $p_{{\mathbb D}}:({\cal G}({\mathbb D}), K)\to ({\cal C}, J)$ (for instance, one can take $p_{{\mathbb D}}$ to be the canonical relative site of $f$, in the sense of Definition \ref{defrelativesiteofamorphism} - see Theorem \ref{thm:relativesitegeometricmorphism} below).

\end{thm}\qed

Let $f:\Ftopos\rightarrow\Etopos$ be a geometric morphism. As observed in Example 2.1.1. (ii) in \cite{CaramelloZanfa}, there is an $\Etopos$-indexed category $\icat_{f}$ associated with $f$, defined on objects by mapping $E$ of $\Etopos$ to the slice topos $\Ftopos/f^*(E)$, with its transition morphisms being the obvious pullback functors. We can perform this more in general: consider two sites $(\cbicat,J)$ and $(\dbicat,K)$ and a $(J,K)$-continuous functors $A:\cbicat\rightarrow \dbicat$: then we can consider the $\cbicat$-indexed category $\icat_A$\index{$\icat_A$} defined by mapping any $c$ in $\cbicat$ to the slice topos $\Sh(\dbicat,K)/\ell_K(A(c))$, and whose transition morphisms are the obvious pullback functors. Notice that the fibration $\icat_f$ defined from the geometric morphism $f$ is now a particular instance of this, where we take as continuous functors the morphism of sites $f^*:(\Etopos, J\can_\Etopos)\rightarrow (\Ftopos, J\can_\Ftopos)$. Every fibration of this form is in fact a stack:
\begin{prop}[Theorem 8.2.2 \cite{CaramelloZanfa}]
Let $A:(\cbicat,J)\rightarrow (\dbicat,K)$ be a $(J,K)$-continuous functor: then the fibration $\icat_A$ defined above is a $J$-stack. In particular, for every geometric morphism $f:\Ftopos\rightarrow\Etopos$ the $\Etopos$-indexed category $\icat_{f}$ of Example 2.1.1. \cite{CaramelloZanfa} is a $J\can_\Etopos$-stack.
\end{prop}
\begin{proof}
	Notice that $\icat_A$ corresponds to the composite pseudofunctor
	\[
	\cbicat\op\xrightarrow{A\op} \dbicat\op\xrightarrow{\canst_{(\dbicat,K)}}\CAT,
	\]
	where $\canst_{(\dbicat,K)}$ denotes the canonical stack for the site $(\dbicat,K)$ (see Definition 2.6.4. and Theorem 2.6.8. in \cite{CaramelloZanfa}).Then we can exploit the notion of direct image of fibrations, introduced in Section 3.1 of \cite{CaramelloZanfa}, and the fact that the direct image along a continuous functor maps stacks to stacks (Proposition 3.4.2.), to conclude that $\icat_A$ is a $J$-stack.
\end{proof}
For a geometric morphism $f:\Ftopos\rightarrow \Etopos$, the fibration associated to $\icat_f$ is made as follows: objects over $E$ in $\Etopos$ are arrows $[u:U\rightarrow f^*(E)]$ of $\Ftopos$, and morphisms $(e,a):[u':U'\rightarrow f^*(E')]\rightarrow [u:U\rightarrow f^*(E)]$ are indexed by two arrows $e:E'\rightarrow E$ and $a:U'\rightarrow U$ making the diagram
\[
\begin{tikzcd}
	U'\ar[d, "u'"] \ar[r, "a"] & U\ar[d,"u"]\\
	f^*(E') \ar[r, "f^*(e)"'] & f^*(E)
\end{tikzcd}\]
commutative. This fibration acts by pullback along the image of arrows coming from the basis, hence cartesian arrows of $\gbicat(\icat_f)$ are those such that the square above is a pullback square.

Notice that the fibration $\gbicat(\icat_\Ftopos)$ corresponds in fact with the comma category $\comma{1_\Ftopos}{f^*}$: we have already met this kind of category in Theorem 3.16. \cite{denseness}, where we showed that any geometric morphism induced by a morphism of sites $A:(\cbicat,J)\rightarrow (\dbicat,K)$ can be described as the geometric morphism induced by the fibration $\comma{1_\dbicat}{A}\rightarrow \cbicat$ upon endowing the domain with a suitable Grothendieck topology $\widetilde{K}$. Applying that in our specific case, we obtain the following result:

\begin{thm}\label{thm:relativesitegeometricmorphism}
	Let $f:\Ftopos\rightarrow \Etopos$ be a geometric morphism: then there exists a Grothendieck topology $J_f$\index{$J_f$} over $\gbicat(\icat_f)$ such that the two toposes $\Ftopos$ and $\Sh(\gbicat(\icat_f), J_f)$ are equivalent as $\Etopos$-toposes: more specifically, a family $\{(a_i, e_i):[u_i:U_i\rightarrow f^*(E_i)]\to[u:U\rightarrow f^*(E)]\ |\ i\in I\}$ is $J_f$-covering if and only if the family $\{a_i:U_i\rightarrow U\ |\ i\in I\}$ is epimorphic in $\Ftopos$.
	
	Thus the geometric morphism $f:\Ftopos\rightarrow\Etopos$ presents $\Ftopos$ as a \emph{topos of relative sheaves over the stack $\icat_f$}:
	\[
	\Ftopos\simeq \Sh(\gbicat(\icat_f), J_f)=:\Sh_{\Etopos}(\icat_f, J_f).
	\]
	We call $({\cal G}(\icat_f),J_f)$ the \emph{canonical relative site} of $f$, and $J_f$ its \emph{canonical relative topology}.
\end{thm}
\begin{proof}
	We apply Theorem 3.16. of \cite{denseness} to the morphism of sites \[f^*:(\Etopos, J\can_\Etopos)\rightarrow (\Ftopos,J\can_\Ftopos):\] the category $\comma{1_\Ftopos}{f^*}$ can be endowed with a topology $\widetilde{J\can_\Ftopos}$ such that $\pi_\Etopos:\comma{1_\Ftopos}{f^*}\rightarrow \Etopos$ is a comorphism of sites and $\pi_\Ftopos:\comma{1_\Ftopos}{f^*}\rightarrow \Ftopos$ induces an equivalence of toposes making the diagram
	\[
	\begin{tikzcd}
		\Ftopos \ar[d, "f"']  \ar[r, "\sim", no head] &{\Sh(\Ftopos,J\can_\Ftopos)} \ar[d, "\Sh(f^*)"']\ar[r, no head, "\sim"]& {\Sh(\comma{1_\Ftopos}{f^*}, \widetilde{J\can_\Ftopos})}\ar[dl, "C_{\pi_\Etopos}"]\\
		\Etopos\ar[r, "\sim", no head] &
		{\Sh(\Etopos,J\can_\Etopos)}&
	\end{tikzcd}
	\] 
	commutative. Setting $J_f:=\widetilde{J\can_\Ftopos}$ concludes the proof.
\end{proof}

Intuitively, in the comma category $\comma{1_\Ftopos}{f^*}$ representing $\icat_f$ as a fibration, the topology $J_f$ says that only the components following $\cal F$ have some ``weight'', which hints that we recover the topos $\cal F$ when we go to the category of sheaves on it; whereas the components following $\cal E$ enables data to be well organized in respect with this basis. Moreover, the projection $\pi_{\cal E}$ onto $\cal E$ has a right adjoint that sends an object $e$ of $E$ to $1_{f^*(e)}: f^*(e) \to f^*(e)$, which obviously acts as $f^*$ on the first component (the one that has a ``weight'' at the topos level). Since a left adjoint being a comorphism is equivalent to a right adjoint being a morphism of sites, and that they induce the same geometric morphism, it is not surprising that we recover the geometric morphism $f$ as induced by the comorphism $\pi_{\cal E}$, at the topos level.

We can consider, more generally, geometric morphisms induced by arbitrary morphisms of sites: 

\begin{defn}\label{defrelativesiteofamorphism}
	Let $A:(\cbicat,J)\rightarrow (\dbicat,K)$ be a morphism of small-generated sites. The \emph{relative site of $A$}\index{site! relative - of a morphism of sites} is the site $({\mathbb I}_{A}, J^{K}_{A})$ (where $J^{K}_{A}$\index{$J^{K}_{A}$} is the topology $\widetilde{K}$ of Theorem 3.16. in \cite{denseness}), together with the canonical projection functor $\pi_{A}:{\mathbb I}_{A} \to {\cal C}$, which is a comorphism of sites $({\mathbb I}_{A}, J^{K}_{A}) \to ({\cal C}, J)$.
\end{defn}

Let $A:(\cbicat,J)\rightarrow (\dbicat,K)$ be a morphism of small-generated sites. Then the site $({\mathbb I}_{A}, J^{K}_{A})$ is small-generated. In particular, the relative site of a geometric morphism is always small-generated.

Recall that by Theorem 3.16. \cite{denseness}, we have the following generalization of Theorem \ref{thm:relativesitegeometricmorphism}:

\begin{thm}
	Let $A:(\cbicat,J)\rightarrow (\dbicat,K)$ be a morphism of small-generated sites. Then the geometric morphism $\Sh(A)$ induced by $A$ coincides with the structure geometric morphism $C_{\pi_{A}}$ associated with the relative site $\pi_{A}:({\mathbb I}_{A}, J^{K}_{A})\to ({\cal C}, J)$. 
\end{thm}

As it can be naturally expected, these canonical sites are small-generated:

\begin{prop}\label{prop:relativesitesmallgenerated}
	Let $A:(\cbicat,J)\rightarrow (\dbicat,K)$ be a morphism of small-generated sites. Then the site $({\mathbb I}_{A}, J^{K}_{A})$ is small-generated.
	
	In particular, the relative site of a geometric morphism is always small-generated.
\end{prop}

\begin{proof}
    Our goal is to build a small $J^K_A$-dense subcategory $\ibicat_A$ of the fibration $\icat_A:=\comma{1_{\Sh(\dbicat,K)}}{\ell_K A}$, \ie 
    such that every object of $\icat_A$ has a $J^K_A$-covering family whose domains all lie in $\ibicat_A$. Notice that a sieve is $J^K_A$-covering if and only if its image in $\Sh(\dbicat,K)$ is $J\can_{\Sh(\dbicat,K)}$-covering, \ie if and only if it contains a small epimorphic family.
    
    To do so, let us set an object $(F,X, f:F\rightarrow \ell_K(A(X)))$ of $\icat_A$.
	Let us denote by $\abicat$ the small $J$-dense subcategory of $\cbicat$ such that $\Sh(\cbicat,J)\simeq \Sh(\abicat, J_{|\abicat})$. Notice that there is a small $J$-covering family $\{y_i:Y_i\rightarrow X\ |\ i\in I\}$ such that all the domains $Y_i$ lie in $\abicat$. Since $\ell_K A$ is cover-preserving, it follows that the family $\{\ell_K(A(y_i)):\ell_K(A(Y_i))\rightarrow \ell_J(A(X))\ |\ i\in I\}$ is $J\can_{\Sh(\dbicat,K)}$-covering, \ie it is jointly epic. We can then consider the pullback squares
	\[
	\begin{tikzcd}
	    F_i \ar[d, "p_i"'] \ar[r, "f_i"] & {\ell_K(A(Y_i))} \ar[d, "\ell_K(A(y_i))"]\\
	    F \ar[r, "f"'] & \ell_K(A(X)): \ar[ul, "\lrcorner" near end, phantom]
	\end{tikzcd}
	\]
	notice that the family $\{p_i:F_i\rightarrow F\ |\ i\in I\}$ is small and jointly epic in $\Sh(\dbicat,K)$, and so the family $\{(p_i,y_i)\ |\ i\in I\}$ is $J^K_A$-covering in $\icat_A$. 
	
	Now, since $\dbicat$ is also small-generated we have a small $K$-dense subcategory $\bbicat$ of $\dbicat$: then every $F_i$ is a colimit of a (small) diagram of representables $\ell_K(D_{i,j})$, $j\in J_i$, such that each object $D_{i,j}$ lies in $\bbicat$. The legs $\lambda_{i,j}:\ell_K(D_{i,j})\rightarrow F_i$ of the colimit cocone are a small jointly epic family in $\Sh(\dbicat,K)$, and thus we have for every $i$ a $J^K_A$-covering family
	\[
	(\lambda_{i,j},1_Y):
	\left(\ell_K(D_{i,j}), Y_i, {f_i\circ \lambda_{i,j}}\right) \rightarrow\left( F_i, Y_i, f_i\right)\]
	in $\icat_A$. Composing each of these families with the family $\{(p_i, y_i)\ |\ i\in I\} $ above, we obtain the $J^K_A$-covering family for $(F,X,f)$ whose domains are all of the form $(\ell_K(B), Z, g:\ell_K(B)\rightarrow \ell_K(A(Z))$ for some $B$ in $\bbicat$ and $Z$ in $\abicat$. We can thus define $\ibicat_A$ to be the full subcategory of $\icat_A$ of objects of this kind: since $\abicat$ and $\bbicat$ are small it has a set of objects, and it is locally small, implying that $\ibicat_A$ is small, and we have just showed that it is $J^K_A$-dense in $\icat_A$. Thus $(\icat_A, J^K_A)$ is small-generated.
 \end{proof}

\subsection{Relative analogue of the Yoneda functor}\label{yonedarel}

The fiber over $E$ of the stack associated to the canonical relative site of a relative topos $f: \cal F \to E$ have been previously described as ${\cal F} \slash f^*E$. These are toposes, and when a relative topos is induced by a site over $({\cal C},J)$, that is a comorphism of sites $p: ({\cal D}, K) \to ({\cal C}, J)$, the fiber of its canonical relative site at $c$ is given by $\Sh({\cal D}, K)\slash C_{p}^{\ast}(l_{J}c)$. We have a site presentation of these fiber toposes: we take $(p \downarrow c , K_c)$, where $K_c$ says that a sieve is covering if and only if its projection on $\cal D$ is a covering one .

Actually, $(p\downarrow c)$ is the category of elements of ${\cal C}(p-,c)$, and $K_c$ is the Giraud topology on it. The correspondence is easily seen with presheaves categories, described as follows: if we have a presheaf $Q$ over $(p\downarrow c)$, for every object $d$ of $\cal D$ we can take the union of all the fibers of $Q$ over arrows coming from $p(d)$: $\amalg_{u: p(d) \to c}Q_{u}$ to be the image of $d$ by a new presheaf in $\cal D$, naturally given with a natural transformation to ${\cal C}(p-,c)$ defined pointwise by sending an object of the sum to its index. In the other direction, if we have a natural transformation from a presheaf $Q$ to ${\cal C}(p-,c)$ we can take its fibers at each arrow $ u: p(d) \to c$, and it gives a presheaf over $(p\downarrow c)$. In general, to see that $\Sh (P,J_P) \simeq \Sh(a_J(P),J_{a_J(P)})$ when the topologies are Giraud's ones, we can show that the unit of the sheafification induces a weakly dense morphism of sites between these sites, concluding by \cite{CaramelloZanfa} Prop 2.10.7.

Following this point of view, the Yoneda embedding sends an arrow $u: p(d) \to c$ to the representable associated to $d$ evaluating its identity on $u$: $\yo_{(p \downarrow c)}(u) = ev_u: \yo_{\cal D}(d) \to \yo_{\cal C}(p(-),c)$. So, the representable presheaves of the fiber over $c$ are the representable presheaves of $\cal D$ seen in all the possible ways over $\yo_{\cal C}(p(-),c)$.

All these sites presenting all the fibers, we can bring them all together as $(p \downarrow 1_{\cal C})$, taking into account all these Yoneda functors on all the fibers, and we obtain a ``Yoneda functor varying through the base'':

\begin{prop}[Theorem 3.20. \cite{denseness}]\label{propdualcomma}
	Let $p:({\cal D}, K)\to ({\cal C}, J)$ be a comorphism of small-generated sites. We have a functor
	\[
	\xi_{p}: (p\downarrow 1_{\cal C}) \to (1_{\Sh({\cal D}, K)} \downarrow C_{p}^{\ast})
	\]
	sending an object $(d, c, \alpha:G(d)\to c)$ of the category $(p\downarrow 1_{\cal C})$ to the object $(l_{K}d, l_Jc, l_Kd\to C_{p}^{\ast}(l_Jc))\cong a_{K}(p^{\ast}(\yo_{\cal D}d))=a_{K}(\Hom_{\cal C}(p(-), c))$ which is the sheafification of the evaluation of the identity of the source representable onto the arrow $\alpha$, well-defined by the Yoneda lemma. This functor is a dense bimorphism of sites
	$((p\downarrow 1_{\cal C}), \overline{K}) \to ((1_{\Sh({\cal D}, K)} \downarrow C_{p}^{\ast}), J_{C_{p}})$ where $\overline{K}$ have for coverings the ones having arrows of the first component being a cover in $\cal D$, and $J_{C_{p}}$ the canonical relative topology defined in  \ref{thm:relativesitegeometricmorphism}.
\end{prop}

It is a dense bimorphism of sites, as the usual Yoneda-sheafification functor is for a non-relative topos. Moreover, we can conserve this property even if we restrict to some ``diagonal'' version of this functor:

\begin{prop}[Proposition 2.5 \cite{fibered}]\label{propbidense}
	Let $p:({\cal D}, K)\to ({\cal C}, J)$ be a comorphism of sites. There is a dense bimorphism of sites 
	\[
	\eta_{p}:({\cal D}, K) \to ((1_{\Sh({\cal D}, K)} \downarrow C_{p}^{\ast}), J_{C_{p}})
	\]
    defined as the composite of the functor $\xi_{p}$ of Proposition \ref{propdualcomma} with the functor $j_{p}:({\cal D}, K) \to ((p\downarrow 1_{\cal C}), \overline{K})$ of Theorem 3.18 \cite{denseness}. Explicitly, $\eta_p$ acts by mapping an object $d$ of $\cal D$ to the triplet $(\ell_K(d), \ell_J(p(d)), x_{d}:\ell_K(d)\to C_{p}^{\ast}(\ell_Jp(d)))\cong \Sheafify_{K}(\cbicat(p(-), p(d)))$, where $x_{d}$ is the sheafification of the arrow $\yo_{\cal D}(d)\to \cbicat(p(-), p(d))$ corresponding to the evaluation on the identity of $p(d)$.
\end{prop}

\begin{remarks}\label{remarkseta}
	\begin{enumerate}[(a)]
	    \item Notice that the image of $\eta_p$ factors through the obvious inclusion $(1_{\Sh({\cal D},K)} \downarrow {C_p^*\ell_J}) \hookrightarrow (1_{\Sh({\cal D},K)} \downarrow {C_p^*})$.
     
        \item If $p$ is faithful then $\eta_{p}$ takes values in the subcategory $(1_{\Sh({\cal D}, K)} \downarrow^{\textup{Sub}} C_{p}^{\ast})$ called the reduced relative site.

		\item When $p=p_{\mathbb P}$ for a fibred preorder site $\mathbb P$ then it is clearly faithful and hence the previous point applies. Moreover, under this hypothesis, the geometric morphism $C_{p}$ is localic and hence the functor $\eta_{p}$ yields a dense bimorphism of sites from $({\cal D}, K)$ to the reduced relative site of $C_{p}$.  
		
		\item  If $p$ is $(K, J)$-continuous, then $\eta_{p}$ sends $d$ in $\cal D$ to the unit of $\lan_{p^{op}} \dashv (-\circ p^{op})$ at $\yo_{\cal D}(d)$. Indeed, under this continuity hypothesis the description of $C_{p}^{\ast}$ simplifies: for any $K$-sheaf $Q$, $C_{p}^{\ast}(Q)=Q\circ p^{\textup{op}}$. Since that by definition $x_d$ is the sheafification of the unit of $\lan_{p\op} \dashv p^*$ at $\yo_{\cal D}(d)$, if $p$ is continuous there is no need to sheafify.
		
		\item If $\cal C$ and $\cal D$ have finite limits and $p$ preserves them, then the functor $\eta_{p}$ also preserves finite limits. Indeed, as the identity and $C_p^*$ preserve finite limits, the comma category is finitely complete and the limits are computed componentwise in it. Since $p$, $C_p^*$ and $l_K$ preserve them, the result is immediate.
	\end{enumerate}
	
\end{remarks}

\section{Relative Diaconescu's theorem}

As explained in the introduction, in this section we shall present three different approaches to the characterization of relative geometric morphisms in terms of functors between sites.

This will culminate, in section 3.2.2, with the proof of a relative version of Diaconescu's equivalence.

\subsection{Approach through the classification theorem}\label{sec:approachDiaconescuclassification}

In this section, we shall investigate how the classification theorems for geometric morphisms over an arbitrary base topos established in section 4 of \cite{CaramelloZanfa} restricts to $2$-natural transformations that are isomorphisms, by using cofinality conditions.

As we saw in section 4.2 \cite{CaramelloZanfa}, given a relative topos $\Etopos \xrightarrow{f} \Sh ({\cal C},J) $ and a comorphism of sites $({\cal D}, K) \xrightarrow{p} ({\cal C}, J)$ we have a $2$-categorical equivalence between the slice categories:
\begin{center}
$\Topos // \Sh ({\cal C},J) ([f],[C_p]) \simeq \Site (({\cal D}, K),(\Etopos,J\can_\Etopos))/f^*\ell_j p$    
\end{center}
 
Let us describe how the correspondence works (recall that two adjoint functors induce adjoints functors between categories of functors, in the same direction by post-composition, and in the reverse direction by precomposotion): 

\begin{itemize}
    \item If we have a diagram 
    \begin{center}
\[\begin{tikzcd}
	{\cal E} && {\Sh({\cal D},K)} \\
	& {\Sh({\cal C},J)}
	\arrow["{C_p}", from=1-3, to=2-2]
	\arrow[""{name=0, anchor=center, inner sep=0}, "f"', from=1-1, to=2-2]
	\arrow["g", dashed, from=1-1, to=1-3]
	\arrow[from=1-3, to=2-2]
	\arrow["\phi"{description}, shorten >=10pt, Rightarrow, dashed, from=1-3, to=0]
\end{tikzcd}\]
    \end{center}
    where $\phi: g^*C_p^* \Rightarrow f^* $ is a natural transformation, we have $C_p^* = a_k (- \circ p^{op})i_k$, and, since $i_k$ and $(- \circ p^{op})$ are two right adjoint functors, we can transpose $\phi$ into $\Tilde{\phi}: g^*a_k \Rightarrow f^*a_J\lan_{p^{op}}$. This transformation is determined, up to isomoprhism, by its action on the generators, and, in light of the good behavior of the left Kan extension with regard to them, we obtain: $\Tilde{\phi}\yo_{\cal D}: g^*l_K \Rightarrow f^*l_Jp$
    \item In the other direction, if we have a natural transformation $\phi$ as in the diagram
    \begin{center}
\[\begin{tikzcd}
	{\cal E} && {({\cal D},K)} \\
	& {({\cal C},J)}
	\arrow["p", from=1-3, to=2-2]
	\arrow["{f^*l_J}", from=2-2, to=1-1]
	\arrow[""{name=0, anchor=center, inner sep=0}, "A"', dashed, from=1-3, to=1-1]
	\arrow["\phi"', shorten <=3pt, Rightarrow, dashed, from=0, to=2-2]
\end{tikzcd}\]
    \end{center}
    we obtain, by pre-composition, a natural transformation $( - \circ \phi^{op}): (- \circ p^{op})(- \circ (f^*l_J)^{op}) \Rightarrow (- \circ A^{op}) $ at the level of presheaf categories. Applying sheafification and inclusion yields a natural transformation at the level of the sheaf categories: $a_K( - \circ \phi^{op})i_{\cal E}: a_K(- \circ p^{op})(- \circ (f^*l_J)^{op})i_{\cal E} \Rightarrow a_K(- \circ A^{op})i_{\cal E} $ where $(i_{\cal E} \vdash a_{\cal E})$ is the ajunction corresponding to the canonical topology on $\cal E$. Since $p$ is a comorphism, pre-composition by $p$ is leaved unchanged by sheafification on the source. Recall also that $f^*l_J$ is a morphism of sites for the canonical topology of $\cal E$, inducing the geometric morphism $f$. As a consequence, we have $a_K( - \circ \phi^{op})i_{\cal E}: C_p^*f_* \Rightarrow A_* $. By transposing twice, we obtain $\overline{a_K( - \circ \phi^{op})i_{\cal E}}: \Sh(A)^*C_p^* \Rightarrow f^* $.
\end{itemize}

In order to characterize morphisms at the site level inducing geometric morphisms \textit{over} $\Sh({\cal C},J)$, we want to restrict to those for which the induced $\tilde{\phi}$ between the inverse images are natural isomorphisms. For this, we shall describe the subcategories of the left-hand slice category of geometric morphisms in terms of the associated functors between relative sites, by representing $f$ and $C_{p'}$ where $p'$ is the canonical stack of $f$ over $\cal C$. Note that, by the universal property of the comma category $(1\downarrow f^{\ast}l_{J})$, giving a functor $A:{\cal D}\to {\cal F}$ together with a natural transformation $\phi:A \to f^{\ast}l_{J}p$ corresponds precisely to giving a functor ${\cal D}\to (1\downarrow f^{\ast}l_{J})$ over $\cal C$:

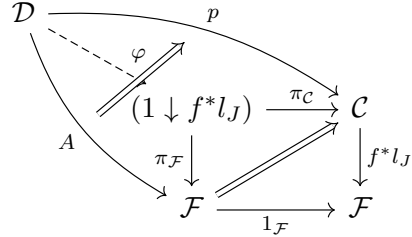
\[\begin{tikzcd}
	{\cal D} \\
	& {(1 \downarrow f^*l_J)} & {\cal C} \\
	& {\cal F} & {\cal F}
	\arrow[dashed, from=1-1, to=2-2]
	\arrow["{1_{\cal F}}"', from=3-2, to=3-3]
	\arrow["{f^*l_J}", from=2-3, to=3-3]
	\arrow["{\pi_{\cal F}}"', from=2-2, to=3-2]
	\arrow["{\pi_{\cal C}}", from=2-2, to=2-3]
	\arrow[""{name=0, anchor=center, inner sep=0}, "p", bend right=-18, from=1-1, to=2-3]
	\arrow[""{name=1, anchor=center, inner sep=0}, "A"', bend left=-20, from=1-1, to=3-2]
	\arrow[Rightarrow, from=3-2, to=2-3]
	\arrow["\phi"{pos=0.6}, shorten <=7pt, shorten >=7pt, Rightarrow, from=1, to=0]
\end{tikzcd}\]

See Theorem \ref{thm:RelativeDiaconescumorphismsrelativesites} below for the precise correspondence between such functors and the relative geometric morphisms which they induce.

Interestingly, if the comorphism $p'$ representing $f$ is not the canonical one (arising from the canonical stack of $f$), there may be functors $A:{\cal D}\to {\cal D}'$, together with a natural transformation $\phi: p' \circ A \Rightarrow p$, inducing morphisms of toposes $C_{p'}\to C_{p}$ over $\Sh({\cal C}, J)$ without $\phi$ being an isomorphism. This motivates the introduction of the following notions: 

\begin{defn}\label{defmorphismrelsites}
Let $({\cal C}, J)$ be a small-generated site, and $p:({\cal D}, K)\to ({\cal C}, J)$ and $p':({\cal D'}, K')\to ({\cal C}, J)$ two comorphisms of sites. Let $(A, \phi)$ be a pair consisting of a morphism of sites $A:({\cal D}, K)\to ({\cal D}', K')$ and a natural transformation $\phi:p'\circ A \Rightarrow p$.

\begin{enumerate}[(a)]
    \item We say that $(A, \phi)$ is a \emph{lax morphism of sites over $({\cal C},J)$} if $C_p \circ \Sh(A) \simeq C_{p'}$.
    \item We say that $(A, \phi)$ is a \emph{morphism of sites over $({\cal C}, J)$} if $\phi$ is an isomorphism, and $C_p \circ \Sh(A) \simeq C_{p'}$.

    \item If $p$ and $p'$ are fibrations endowed with Grothendieck topologies containing the respective Giraud ones, that is, relative sites over $({\cal C}, J)$, we say that $A$ is a \emph{morphism of relative sites} if $A$ is an ordinary morphism of sites $({\cal D}, K)\to ({\cal D}', K')$ which is moreover a morphism of fibrations (i.e. $\phi$ is an isomorphism and it sends cartesian arrows to cartesian arrows).
\end{enumerate}

Indeed, we are interested in exhibiting necessary and sufficient conditions for a morphism of sites $A:({\cal D}, K)\to ({\cal D'}, K')$, together with a natural transformation $\phi: p' \circ A \Rightarrow p$ to be a morphism of sites \textit{over} $({\cal C}, J)$, \textit{i.e.} to induce a morphism of relative toposes $\Sh(A):[C_{p}] \to [C_{p'}]$.

\end{defn}
    
As in the case of the above correspondence, the natural transformation $\phi:p'\circ A \Rightarrow p$ induces a natural transformation $\Tilde{\phi}: (C_{p}\circ \Sh(A))^{\ast}=\Sh(A)^{\ast}\circ C_{p}^{\ast} \Rightarrow C_{p'}^{\ast}$, defined as follows: Starting from
\[
\phi^*: p^*\Rightarrow A^*\circ p'^*,
\]
defined as pre-composition with $\phi\op$, we consider its componentwise mate 
\[
\tilde{\phi}:\lan_{A\op}\circ p^*\Rightarrow p'^*
\]
and then we compose it with $\Sheafify_{K'}$, obtaining a natural transformation \[
\Tilde{\phi}:\Sh(A)^*\circ C_p^*\circ \Sheafify_J\simeq \Sheafify_{K'}\circ\lan_{A\op}\circ p^*\xRightarrow{\Sheafify_{K'}\circ \tilde{\phi}} \Sheafify_{K'}\circ p'^*\simeq C_{p'}^*\circ \Sheafify_J. 
\]
Notice that the canonical isomorphisms appearing are due to the fact that $A$ is a morphism of sites (or, more precisely, that it is $(K,K')$-continuous), while $p$ and $p'$ are comorphisms of sites.

Thus, for each $c$ in $\cbicat$, we have the component
\[
\Tilde{\phi}(\ell_J(c)):\Sh(A)^{\ast}(C_{p}^{\ast}(\ell_J(c)))) \to C_{p'}^{\ast}(\ell_J(c)),
\] 
while the other components of $\Tilde{\phi}$ can be retrieved by considering each $J$-sheaf as a colimit of representable presheaves.

We want to characterize the pairs $(A, \phi)$ such that $\Tilde{\phi}$ is an isomorphism. For this, we shall find it convenient to resort to the theory of relative cofinality developed in \cite{denseness}. The key result that we are going to use is the following:

\begin{prop}[Proposition 2.21. \cite{denseness}]\label{procofinality}
	Let $({\cal C}, J)$ be a small-generated site and $F:{\cal A}\to {\cal C}$ and $F':{\cal A}'\to {\cal C}$ two functors to $\cal C$ related by a functor $\xi:{\cal A}\to {\cal A}'$ and a natural transformation $\alpha:F\Rightarrow F'\circ \xi$. Then the canonical arrow 
	\[
	\tilde{\alpha}: \colim_{[{\cal C}\op, \Set]}(\yo_{\cal C}\circ F) \to \colim_{[{\cal C}\op, \Set]}(\yo_{\cal C}\circ F')
	\]
	defined above is sent by $\Sheafify_{J}$ to an isomorphism 
	\[
	\Sheafify_{J}(\tilde{\alpha}): \colim_{\Sh({\cal C}, J)}(\ell_J\circ F) \to \colim_{\Sh({\cal C}, J)}(\ell_J\circ F') 
	\]
	if and only if $(\xi, \alpha)$ satisfies the following \ac cofinality' conditions:
	\begin{enumerate}[(i)]
		\item For any object $X$ of $\cal C$ and any arrow $x:X\to F'(A')$ in $\cal C$ there are a $J$-covering family $\{y_i: Y_i \to X \mid i\in I\}$ and for each $i\in I$ an object $A_{i}$ of $\cal A$ and an arrow $p_i:Y_{i}\to F(A_{i})$ such that $x\circ y_{i}$ and $\alpha(A_{i})\circ p_{i}$ belong to the same connected component of $\comma{Y_i}{F'}$.
		
		\item For any object $X$ of $\cal C$ and any arrows $x:X\to F(A)$ and $x':X\to F(B)$ in $\cal C$ such that $\alpha(A)\circ x$ and $\alpha(B)\circ x'$ belong to the same connected component of $\comma{X}{F'}$ there is a $J$-covering family $\{y_i: Y_i \to X \mid i\in I\}$ such that $x\circ y_i$ and $x'\circ y_i$ belong to the same connected component of $\comma{Y_i}{F}$. 
	\end{enumerate}	
\end{prop}

In order to apply it, we need to reformulate each component $\tilde{\phi}(\yo(c))$ as an arrow between colimits in the presheaf topos $[{\cal D}'^{op},\Set]$, and then express the condition for its $K'$-sheafification $\bar{\phi}(\ell_J(c))$ to be an isomorphism. 

First of all, notice that $p'^*\yo(c)=\cbicat(p'(-),c)$ is the colimit
\[
\colim\left(\comma{p'}{c}\xrightarrow{\pi^{p'}_c} {\cal D}' \xrightarrow{\yo_{D'}}[{\cal D}'^{op},\Set]\right)=\colim_{w:p'(d')\to c} \yo(d'),
\]
since $\comma{p'}{c}$ is the Grothendieck fibration associated to $\cbicat(p'(-),c)$. Similarly, the presheaf $\lan_{A\op}\circ p^*(\yo(c))$ is the colimit
\[
\colim\left( \comma{p}{c}\xrightarrow{\pi^p_c}\dbicat\xrightarrow{A} {\cal D}'\xrightarrow{\yo_{{\cal D}'}}[{\cal D}'^{op},\Set]\right)\cong \colim_{v:p(d)\to c} \yo(A(d)): 
\]
this can be shown exploiting the commutativity of $\lan_{A\op}$ with colimits and the natural isomorphism $\yo_{{\cal D}'}\circ A\cong \lan_{A\op}\circ \yo_\dbicat$. The arrow $\tilde{\phi}(\yo(c)):\lan_{A\op}\circ p^*(\yo(c))\to p'^*(\yo(c))$, being an arrow between colimits, can be induced at the level of diagrams by considering the commutative diagram
\[
\begin{tikzcd}
    {\comma{p}{c}}  \ar[d, "\pi^p_c"] \ar[r, "A_c"] & {\comma{p'}{c}} \ar[d, "\pi_c^{p'}"] \\
    \dbicat  \ar[r, "A"] & {\cal D}',
\end{tikzcd}
\]
where $A_c$ maps an object $(d, v:p(d)\to c)$ of $\comma{p}{c}$ to the object $(A(d), v\circ \phi_d: p'A(d)\to p(d)\to c)$ of $\comma{p'}{c}$. Thus, we can conclude the following:
\begin{thm}\label{thmcofinality}
    Let $(\cbicat,J)$ be a small-generated site, $p:(\dbicat,K)\to (\cbicat,J)$ and $p':({\cal D}',K')\to (\cbicat,J)$ two comorphisms, $A:(\dbicat,K)\to ({\cal D}',K')$ a morphism of sites and $\phi: p'\circ A\Rightarrow p$ a natural transformation. Then the following are equivalent:
    \begin{enumerate}[(i)]
        
        \item the pair $(A,\phi)$ induces a relative geometric morphism
        \[
        \Sh(A): [C_p:\Sh(\dbicat,K)\to \Sh(\cbicat,J)] \to [C_{p'}:\Sh({\cal D}',K')\to \Sh(\cbicat,J)];
        \]       
        
        \item the functors 
        \[
        A_c:\comma{p}{c}\to \comma{p'}{c},
        \]
        seen as functors over ${\cal D}'$, together with the identities $\pi_c^{p'}\circ A_c=A\circ \pi_c^p$ satisfy the simplified conditions of Proposition \ref{procofinality}: that is, 
        
        \begin{enumerate}[(a)]
            \item Local surjectivity of the canonical arrow between the colimits: for every arrow $p'(d') \xrightarrow{u} c $, there exist a covering sieve $S$ on $d'$ and, for every arrow $d'_i \xrightarrow{f_i} d'$ in $S$, a pair of arrows $(d'_i \xrightarrow{v_i} A(k_i),p(k_i) \xrightarrow{u_i} c)$ such that $u \circ p'(f_i) = u_i \circ \phi_{k_i} \circ p'(v_i)$.
            \item Local injectivity of the canonical arrow between the colimits: for every two pairs of arrows $(d' \xrightarrow{v} A(k), p(k) \xrightarrow{u} c)$ and $(d' \xrightarrow{v'} A(k'), p(k') \xrightarrow{u'} c)$ such that $u \circ \phi_k \circ p'(v) = u' \circ \phi_{k'} \circ p'(v')$, there exists a covering sieve $S$ on $d'$ such that, for every $d'_i \xrightarrow{f_i} d'$ in $S$: $v \circ f_i $ and $ v' \circ f_i$ are in the same connected components of the category $(d'_i \downarrow A\pi_c)$.
            \end{enumerate}
   
    \end{enumerate}

\end{thm}

\begin{proof}
The target $\colim_{w:p'(d')\to c} \yo(d')$ of $\Tilde{\phi}$ is just given by the co-Yoneda lemma. By construction, in such a colimit computed by sums and coequalizers, every class of ``$x \xrightarrow{f} d' $ in the summand indexed by $p'(d') \xrightarrow{u} c$'' is equal to the class of  ``$1_{x}$ in the summand indexed by $p'(x) \xrightarrow{p'(f)} p'(d') \xrightarrow{u} c$'', and two such classes are equal exactly when, viewed as a such class, their summand's indexes are equal. 

In the light of this formulation of the target colimit, the local surjectivity condition states that every element of the target functor is locally reached by $\Tilde{\phi}$. And the local injectivity condition is the reformulation of the fact that if two classes of the source colimits, that is to say two data of ``an arrow $x \to A(k)$ with an index summand $p(k) \to c$'' are sent to the same arrow $p'(x) \to p'A(k) \to p(k) \to c$ by the natural transformation, then they should already be equal in this source colimit, \textit{i.e.} being locally in the same connected component. 
\end{proof}

\begin{remarks}\label{remcof}

\begin{enumerate}[(a)]
    \item The advantage of this characterization is its generality, we do not need $A$ to be a morphism of sites: we just used the commutativity of $\Sh(A)^*$ with colimits. Hence, it could specialize to weaker classes of functor, for example the continuous ones, stating that we have a relative continuous functor if he satisfies these cofinality conditions together with the continuity conditions exposed in \cite{denseness} Prop 4.13.
    \item Note that the local surjectivity condition says that the categories $((p' \downarrow c) \downarrow A_c)$ are locally non-empty. Also, under the assumption that $A$ is a morphism of sites, the local injectivity condition simplifies. This criteria becomes, in fact, equivalent to: the categories $((p' \downarrow c) \downarrow A_c)$ locally have a cone over every diagram of product shape. Indeed, two arrows that can be linked by a product diagram are, de facto, linked, and if we have two arrows that are locally in the same connected component, we can, by the filtering characterization of $A$, link them by a product diagram in $(d' \downarrow A\pi^p_c)$. Also, the fact that the categories $((p' \downarrow c) \downarrow A_c)$ locally have the diagrams of equalizer shape is equivalent to the ``absolute'' analogue, as an arrow going to the source of two arrows always comes with an obvious index. We will explore this kind of conditions more in detail in the next part; this will give rise to the notion of \textit{relative} local filteredness (cf. Proposition \ref{propmorphimrelmorphism} below).

    \item One can show that if $A$ is a morphism of sites as well as a morphism of fibrations then the conditions of Theorem \ref{thmcofinality} are satisfied, by using the explicit characterization of morphisms of sites provided by Definition 3.2 \cite{denseness} and the factorization of any arrow in a fibration as a vertical arrow followed by a horizonal one (cf. also Proposition \ref{propmorphismfibrationsrelmorphismsites} below). 
\end{enumerate}
\end{remarks}

\subsection{Generalizing Giraud's approach}

In this section, we generalize the approach of Giraud, first in the cartesian setting and then in the general one, with the aim of obtaining a characterization of the morphisms of sites inducing morphisms of relative toposes as those which yield cartesian morphisms of stacks satisfying suitable properties. For this, we shall use the theory of $\eta$-extensions of morphisms of sites, in order to move from non-necessarily cartesian sites to cartesian ones, for which the relative flatness condition can be expressed in terms of finite-limite preservation conditions. This parallels what happens in the \ac absolute', where the condition for a functor to be flat is expressed as the requirement for the associated (left Kan) extension along the Yoneda embedding to preserve finite limits.

Recall that, in \cite{fibered}, a relative Diaconescu's theorem in the cartesian setting was proved:

\begin{thm}[Theorem 3.3  \cite{fibered}]\label{thm:RelativeDiaconescuCartesian}
	Let $({\cal C}, J)$ be a small-generated site, where $\cal C$ is a cartesian category, ${\mathbb D}:{\cal C}^{\textup{op}} \to \Cat$ a cartesian pseudofunctor, $K$ a Grothendieck topology on ${\cal G}({\mathbb D})$ containing Giraud's topology $J_\dcat$, $A:{\cal C} \to \Ftopos$ a cartesian $J$-continuous functor inducing a geometric morphism $f:\Ftopos\to \Sh({\cal C}, J)$. Then, considering $p_{\mathbb D}$ as a comorphism of sites $({\cal G}({\mathbb D}), K) \to ({\cal C}, J)$, we have an equivalence of categories
	\[
	\Topos/{\Sh({\cal C}, J)}([f], [C_{p_{\mathbb D}}])\simeq \Fib_{\cal C}\cartcov((\gbicat(\dcat), K),(\comma{1_\Ftopos}{A} , J_{f}|_{\comma{1_\Ftopos}{A}})),
	\]
	where $\Fib_{\cal C}\cartcov((\gbicat(\dcat), K),(\comma{1_\Ftopos}{A} , J_{f}|_{\comma{1_\Ftopos}{A}}))$\index{$\Fib\cartcov_\cbicat$} is the category of morphisms of fibrations over $\cal C$ which are cartesian at each fibre and cover-preserving.
\end{thm}

In the absolute setting one takes the left Kan extension along the Yoneda embedding to associate to a functor $A:{\cal C}\to \Ftopos$ with values in a Grothendieck topos $\Ftopos$ a pair of adjoint functors $(L_{A}:[{\cal C}\op, \Set]\to \Ftopos\dashv R_{A}:\Ftopos \to [{\cal C}\op, \Set])$, and understands the flatness of $A$ in terms of the property of the left Kan extension $L_{A}$ to preserve finite limits, that is, to yield a morphism of cartesian sites $([\cbicat\op,\Set], \widehat{J})\to (\Ftopos, J\can_{\Ftopos})$. It is natural to wonder if one can reproduce this method in the relative setting, in order in particular to characterize the condition on a functor to be a relative morphisms of sites in terms of finite-limit preservation conditions of a suitable extension of it. In this section we shall see that it is indeed possible to achieve this.

We recall that we defined an appropriate relative analogue of the Yoneda embedding in the relative setting, provided by the functor $\eta$ of Proposition \ref{propbidense}. The following construction will provide the good notion of extension of a lax morphism of sites over a base one along this functor; applying it in the context of relative sites yields a \ac relative' analogue of the left Kan extension functor operation along the Yoneda embedding.

As in the previous section, let $p:({\cal D}, K)\to ({\cal C}, J)$ and $p':({\cal D}', K')\to ({\cal C}, J)$ two comorphisms of sites, $A: ({\cal D},K) \to ({\cal D}', K')$ a morphism of sites, and $\phi: p' \circ A \Rightarrow p$ a natural transformation. As we have seen, $A$ induces a geometric morphism $\Sh(A):\Sh({\cal D}', K')\to \Sh({\cal D}, K)$ and we have a natural transformation $\Tilde{\phi}:\Sh(A)^{\ast}\circ C_{p}^{\ast}\to C_{p'}^{\ast}$. 

This allows us to define a functor
\[
\tilde{A}:(1_{\Sh({\cal D}, K)} \downarrow C_{p}^{\ast}) \to  (1_{\Sh({\cal D}', K')} \downarrow C_{p'}^{\ast})
\]
as follows: $\tilde{A}(G, F, g:G\to C_{p}^{\ast}(F))=(\Sh(A)^{\ast}(G), F, \Tilde{\phi} \circ \Sh(A)^{\ast}(g))$. We call this functor the \emph{$\eta$-extension of $A$}.

For any object $c$ of ${\cal C}$, we shall denote by $\tilde{A}(c)$ the fiber 
\[
\Sh({\cal D}, K)\slash C_{p}^{\ast}(l_{J}(c)) \to \Sh({\cal D}', K')\slash C_{p'}^{\ast}(l_{J}(c)) 
\]
of the functor $\tilde{A}$ at the object $l_J(c)$. 

The following proposition summarizes the main properties of this construction:

\begin{prop}\label{etaextensionproperties}
	Let $(A, p, p', \phi)$ be as above, and $(1_{\Sh({\cal D}, K)} \downarrow C_p^*)$ and $(1_{\Sh({\cal D'}, K')} \downarrow C_{p'}^*)$ being endowed with their canonical relative topologies. Then:
	\begin{enumerate}[(i)]
		\item The following diagram commutes:
		\[\begin{tikzcd}
			{(1_{\textup{\bf Sh}({\cal D}, K)} \downarrow C_{p}^{\ast})} && {(1_{\textup{\bf Sh}({\cal D}', K')} \downarrow C_{p'}^{\ast})} \\
			{\textup{\bf Sh}({\cal D}, K)} && {\textup{\bf Sh}({\cal D}', K')}
			\arrow["{\tilde{A}}", from=1-1, to=1-3]
			\arrow["{\textup{\bf Sh}(A)^{\ast}}", from=2-1, to=2-3]
			\arrow["{\pi_{\textup{\bf Sh}({\cal D}, K)}}", from=1-1, to=2-1]
			\arrow["{\pi_{\textup{\bf Sh}({\cal D}', K')}}", from=1-3, to=2-3]
		\end{tikzcd}\]
	
		\item The $\eta$-extension $\tilde{A}$ is a morphism of sites.
		
		\item The natural transformation $\Tilde{\phi}$ being an isomorphism, that is ${\bf Sh} (A)$ being a morphism of relative toposes, is exactly characterized by the following equivalent statements:
        \begin{enumerate}[(a)]
            \item The functors $\Tilde{A}(c)$ are cartesians (i.e. $\Tilde{A}$ preserves finite limits \emph{at each fiber}, not just globally).
            \item  The functor $\tilde{A}$ is a morphism of fibrations.
        \end{enumerate}

        \item We have an induced natural transformation $\Phi$: 
        \[\begin{tikzcd}
    	{(1_{\Sh ({\cal D},K)} \downarrow C_p^*}) & ({1_{\Sh ({\cal D}',K')} \downarrow C_p'^*}) \\
    	{\cal D} & {{\cal D}'}
    	\arrow["{\tilde{A}}", from=1-1, to=1-2]
   	\arrow["{\eta_{\cal D}}", from=2-1, to=1-1]
    	\arrow["{\eta_{{\cal D}'}}"', from=2-2, to=1-2]
    	\arrow["A"', from=2-1, to=2-2]
    	\arrow["\Phi"{description}, shorten <=12pt, shorten >=8pt, Rightarrow, from=2-2, to=1-1]
        \end{tikzcd}\]
        This square commutes up to Morita-equivalence, and if $\phi$ is an isomorphism then $\Phi$ also is.
		
		\item The functor $\tilde{A}$ can be characterized, by using the universal property of the comma category $(1_{\textup{\bf Sh}({\cal D}', K')} \downarrow C_{p'}^{\ast})$, by the two functors  $\Sh(A)^*\pi_{\Sh({\cal D},K)}$ and $\pi_{\Sh({\cal C},J)}$, and the natural transformation given in $F \xrightarrow{w} C^*_pF $ by $\Sh (A)^* G \xrightarrow{\Sh(A)^*w} \Sh(A)^*C_p^*F \xrightarrow{\Tilde{\phi}_F} C_{p'}^*F $
		\end{enumerate}
\end{prop}

\begin{proof}
	\begin{enumerate}[(i)]
        \item Immediate checking by the fact that the first component of the action of $\Tilde{A}$ is given by the action of $\Sh (A)^*$ on the first component.
 
	    \item By definition, the covering sieves in the mentioned comma categories are given by those whose first components are covering. But the first components of their image by $\Tilde{A}$ is just their image by $\Sh (A)^*$, which is cover preserving, and so $\Tilde{A}$ is cover preserving. Since the identity functors, $C_p^*$ and $C_{p'}^*$ are finite limit preserving, the comma categories are finitely complete, and their limits are computed component-wise. Therefore, the flatness condition is just the preservation of limits by $\Tilde{A}$, which is immediate by their component-wise computation and the finite limit-preservation of $\Sh(A)^*$, $C_p^*$ and $C_{p'}^*$. 

        \item \begin{enumerate}[(a)]
            \item Recall that we can reduce to check if $\Tilde{\phi}$ is an isomorphism at the level of generators. Note also that the terminal object of the fiber at $c$ of the source fibration $\Sh({\cal D}, K)\slash C_{p}^{\ast}(l_{J}(c))$ is the identity on $l_{J}(c)$, which is sent on $\Tilde{\phi}_c$ by $\Tilde{A}$. Hence, the preservation of the fiberwise finite limits by $\Tilde{A}(c)$ implies that this terminal object is sent to a terminal one, that is: $\Tilde{\phi}_c$ is an isomorphism. Conversely, if the $\Tilde{\phi}_c$ are isomorphisms, the finite limit preservation at each fiber is given by the preservation of finite limits by $\Sh(A)^*$. 
            \item If $\Tilde{A}$ is a morphism of fibrations, since it respects finite limits after (ii), it hence preserves them \emph{fiberwise}. In particular the terminal objects on the fibers; \emph{i.e.} $\Tilde{\phi}_c$ are isomorphisms. Conversely, recall that the cartesian arrows of the source and target fibrations are the represented that induce pullback squares. If $\Tilde{\phi}$ is an isomorphism, it can be easily seen that $\Tilde{A}$ sends these cartesian squares to cartesian squares, because $\Tilde{\phi}$ induces isomorphisms of pullback diagrams. Therefore, it is a morphism of \emph{fibrations}.
        \end{enumerate}

        \item This property can already be seen at the level of presheaves. Let $\eta^p$ be the unit (at the level of presheaves) of the adjunction $(-\circ p^{op}) \vdash \lan_{p^{op}}$. If we restrict this natural transformation to the representable presheaves, we have an explicit description: $\eta_p \yo_{D}(d): \yo_{D}(d) \to \mathcal{C} ( p - , pd ) $ is the evaluation at $1_{pd}$. 

    The natural transformation $\Phi$ thus have to be between $\lan_{A^{op}} \eta^p \yo_{\cal D} \circ  \Tilde{\phi} \yo_{\cal C} p$ and $\eta^{p'}\yo{\cal D'}A$. On an object $d$ of $\cal D$, it is given by: 
\[\begin{tikzcd}
	{\yo_{\cal D'}(Ad)} && {{\cal C}(p'(-),p'(Ad))} \\
	{\lan_{A^{op}}\yo_{\cal D}} & {\lan_{A^{op}}{\cal C}(p(-),p(d))} & {{\cal C}(p'(-),p(d))}
	\arrow["{\eta^{p'}(Ad)}", from=1-1, to=1-3]
	\arrow["\wr"', from=1-1, to=2-1]
	\arrow["{\phi_d \circ -}", from=1-3, to=2-3]
	\arrow["{\lan_{A^{op}}(\eta^p(d))}"', shift right, from=2-1, to=2-2]
	\arrow["{\tilde{\phi}_{\yo_{\cal C}(p(d))}}"', from=2-2, to=2-3]
\end{tikzcd}\] where we use the identifications $\lan_{A^{op}}\yo_{\cal D} \simeq \yo_{\cal D'}A$. The higher raw is $\eta_{\cal D'} \circ A (d)$, and the lower row is $\tilde{A} \circ \eta_{\cal D} (d)$. The vertical arrows constitute $\Phi_d$, as the commutation can be checked: the higher path sends $1_{Ad}$ to $1_{p'(Ad)}$, and then to $\phi_d$, and for the lower path we can use the presentation of $\lan_{A^{op}}(\eta^p(d))$ as a colimit, as in the cofinality part, which also gives $\phi_d$ after a quick computation. 

We note that, when $\phi$ is an isomorphism, the post-composition with $\phi_d$ also is, and so we have a natural isomorphism $\eta_{\cal D'} \circ A \simeq \tilde{A} \circ \eta_{\cal D}$. But, even if $\phi$ is not an isomorphism, we can see that, at the topos level, the following diagram commutes:

\[\begin{tikzcd}
	&& {\Sh(\cal D' , K')} \\
	{\Sh ((1_{\Sh(\cal D,K)} \downarrow C_p^*),J_{C_p})} && {\Sh ((1_{\Sh(\cal D',K')} \downarrow C_{p'}^*),J_{C_{p'}})} \\
	\\
	{\Sh(\cal D,K)} && {\Sh(\cal D',K')}
	\arrow["{\Sh(\eta_{\cal D})^*}", from=4-1, to=2-1]
	\arrow["{\Sh(A)^*}"', from=4-1, to=4-3]
	\arrow["{\Sh(\eta_{\cal D'})^*}"', from=4-3, to=2-3]
	\arrow["{\Sh(\tilde{A})^*}", from=2-1, to=2-3]
	\arrow["{\Sh(\pi_{\cal D'})^*}", from=2-3, to=1-3]
\end{tikzcd}\]
Indeed, the first components of the actions of the two paths of the square are obviously the same on the generators, and so the projection makes them commute. Now, $\Sh (\pi_{\cal D'})$ is an equivalence, and so the square commutes up to isomorphism. 
     \item The universal property of the comma category $(1_{\Sh ({\cal D'}, K')} \downarrow C_{p'}^*)$ says that a functor going into this category is the same as a diagram of this form:

\[\begin{tikzcd}
	\bullet && {\Sh({\cal C},J)} \\
	\\
	{\Sh({\cal D'},K')} && {\Sh({\cal D'},K')}
	\arrow["{1_{\Sh({\cal D'},K')}}", from=3-1, to=3-3]
	\arrow["{C_{p'}^*}"', from=1-3, to=3-3]
	\arrow[from=1-1, to=3-1]
	\arrow[from=1-1, to=1-3]
	\arrow[Rightarrow, from=3-1, to=1-3]
\end{tikzcd}\]
    The characterisation of $\tilde{A}$ is the data of the two functors is hence given by $\Sh (A)^*\pi_{\Sh({\cal D},K)}$ and $\pi_{\Sh({\cal C},J)}$, the transformation given in $F \xrightarrow{w} C^*_pG $ by $\Sh (A)^* F \xrightarrow{\Sh(A)^*w} \Sh(A)^*C_p^*G \xrightarrow{\Tilde{\phi}_G} C_{p'}^*G $.    
	\end{enumerate}

\end{proof}	

\begin{prop}\label{morphsitesinducing}
In the same setting as Proposition \ref{etaextensionproperties}, the functors $\Tilde{A}(c):  \Sh({\cal D}, K)\slash C_{p}^{\ast}(l_{J}c) \to \Sh({\cal D'}, K')\slash C_{p'}^{\ast}(l_{J}c)$ are induced by the $A_c: (p \downarrow c , K_c) \to (p' \downarrow c , K'_c)$ in the usual way: $\Tilde{A}(c) \simeq a_{K'_c} \lan_{A_c^{op}} i_{K_c}$. 
\end{prop}

\begin{proof}
Recall that, in respect with the site-characterization given in the begging of \ref{yonedarel}, the sheafification $a_{K_c}: \Hat{{\cal D}} \slash {\cal C}(p-,c) \to \Tilde{\cal D}\slash C_p^*l_Jc$ acts by sending an arrow $u: P \to {\cal C}(p-,c)$ to its sheafification $a_K(u): a_K(P) \to a_K({\cal C}(p-,c))$ and the inclusion is given by the pullback along the unit of $i_K \vdash a_K$. We can reduce to the case of the trivial topology, the general case can be deduced from these considerations. 
The following diagram, where vertical arrows are the respective Yoneda embeddings, and $\widehat{A}(c)$ is $\Tilde{A}(c)$ for the trivial topology, is commutative:

\[\begin{tikzcd}
	({\widehat{\cal D} \downarrow {\cal C}(p-,c)}) & ({\widehat{\cal D'} \downarrow {\cal C}(p'-,c)}) \\
	{(p \downarrow c)} & {(p' \downarrow c)}
	\arrow[from=2-1, to=1-1]
	\arrow["{\widehat{A}(c)}", from=1-1, to=1-2]
	\arrow["{A_c}"', from=2-1, to=2-2]
	\arrow[from=2-2, to=1-2]
\end{tikzcd}\]

Indeed, in virtue of the discussion in the begining of \ref{yonedarel} , the Yoneda embedding from $(p \downarrow c)$ sends an arrow $u: p(d) \to c$ to the presheaf $\yo_{\cal D}(d) \to {\cal C}(p-,c)$ evaluating the identity on $u$. Then, it is sent by $\widehat{A}(c)$ to $\yo_{\cal D'}(A(d)) \to \colim_{v: p(d) \to c}\yo_{\cal D'}(A(d)) \to {\cal C}(p'-,c)$ where the first arrow sends the identity to the class of the identity in the summand indexed by $u$, and the second arrow sends it to its pre-composition with $\phi_c$. The other path of the diagram gives: an arrow $u: p(d) \to c$ is sent to $u\phi_c$ by $A_c$, and then to the evaluation of the identity of $\yo_{\cal D'}(A(d))$ to this arrow in $\yo_{\cal D'}(A(d))$, which is the same as the first computation. Hence, since $\widehat{A}(c)$ commutes with colimits and agrees with the left Kan extension on generators, it is in fact the left Kan extension of $A_c$. 
\end{proof}

The following result is a relative generalisation of Diaconescu's theorem (in its statement, for lightening the notation, we denote by $J_f$ the topology $J_{f}|_{(1_{\cal F}\downarrow f^*l_J)}$)

\begin{thm}\label{thm:RelativeDiaconescumorphismsrelativesites}
Let $({\cal C}, J)$ be a small-generated site, $p:({\cal D}, K)\to ({\cal C}, J)$ a comorphism of sites and a geometric morphism $f:{\cal F}\to \Sh({\cal C}, J)$. Then we have an equivalence of categories
\[
\textup{\bf Geom}_{\Sh({\cal C}, J)}([f], [C_{p}])\simeq \textup{\bf MorSitesOver}_{({\cal C}, J)}(({\cal D}, K), (1_{\cal F}\downarrow f^*l_J), J_{f}),
\]
where $\textup{\bf MorSitesOver}_{({\cal C}, J)}(({\cal D}, K), (1_{\cal F}\downarrow f^*l_J), J_{f})$ is the category of morphisms of sites over $({\cal C}, J)$ from $({\cal D}, K)$ to $((1_{\cal F}\downarrow f^*l_J), J_{f})$ and natural transformations between them.
\end{thm}

\begin{proof}
    The arguments are the same as those used in the proof of Theorem \ref{diaconescufibration}, except for the condition on the morphisms in the right-hand side to be morphisms of sites over $({\cal C}, J)$ instead of morphisms of fibrations. 
\end{proof}

\begin{remark}
If $p$ is a fibration then the condition for the functors on the right-hand side of the equivalence to be morphisms of sites over $({\cal C}, J)$ amounts precisely to the requirement that they be morphisms of relative sites (cf. Theorem  \ref{diaconescufibration}). Proposition \ref{etaextensionproperties} shows that, in the absence of this condition, the correct replacement is the requirement for the $\eta$-extension of the given functor to yield a morphism of fibrations. In fact, more generally, a morphism of sites between two fibrations does not need to be a morphism of fibrations in order to induce a morphism of relative toposes, only its $\eta$-extension needs to. 
\end{remark}

\subsubsection{Relative filteredness}

In the ordinary topos-theoretic setting, morphisms of sites are characterized in terms of flat functors, and these in turn can be characterized as \emph{filtering functors}; this latter characterization is important, since the notion of filtering functor can be axiomatized within geometric logic (over the signature of the theory of functors on the given category). The equivalence between flatness and filteredness is a non-trivial theorem; the remarkable fact underlying this equivalence is that the isomorphism conditions for colimits arising from the finite-limit preservation conditions for the left Kan extension of the given functor along the Yoneda embedding (which express flatness) are in fact (globally, though not individually) equivalent to the condition for a number of families associated with the given functor to be epimorphic.  

It is interesting, notably in relation to our goal of developing a fibrational semantics allowing one to lift classical topos-theoretic notions and results from the ordinary to the relative setting, to try to reformulate the conditions for a functor to be a relative morphism of sites in terms of certain arrows of toposes being epimorphisms, and then express such conditions at the level of (relative) sites.

\begin{prop}\label{propcond1relmorphismofsites}
Let $(A, p, p', \phi)$ be as above. Then the following conditions are equivalent:

\begin{enumerate}[(i)]	
	\item For any object $E$ of $\Sh({\cal C}, J)$, the canonical arrow
	 \[ 
 	(\Tilde{\phi}, 1_{E}):\tilde{A}(C_{p}^{\ast}(E), E, 1_{C_{p}^{\ast}(E)})=(\Sh(A)^{\ast}C_{p}^{\ast}(E), E, 1_{\Sh(A)^{\ast}C_{p}^{\ast}(E)}) \to (C_{p'}^{\ast}(E), E, 1_{C_{p'}^{\ast}(E)})
	\]
	is covering for the canonical topology of the fiber (that is, its first component is an epimorphism in the topos $\Sh({\cal D}', K')$).
	
	\item For any object $c\in {\cal C}$, the arrow 
	\[
	(\Tilde{\phi}, 1_{l_Jc}):\tilde{A}(C_{p}^{\ast}l_Jc, l_Jc, 1_{C_{p}^{\ast}l_Jc})=(\Sh(A)^{\ast}C_{p}^{\ast}l_Jc, l_Jc, 1_{\Sh(A)^{\ast}C_{p}^{\ast}l_Jc}) \to (C_{p'}^{\ast}l_Jc, l_Jc, 1_{C_{p'}^{\ast}l_Jc})
	\]
	is covering for the canonical relative topology (that is, the arrow  
	\[
	\Tilde{\phi}_{l_Jc}:\Sh(A)^{\ast}C_{p}^{\ast}l_Jc \to  C_{p'}^{\ast}l_Jc
	\]
	is an epimorphism in the topos $\Sh({\cal D}', K')$).
	
	\item For any $c\in {\cal C}$, the unique arrow 
	\[
	\tilde{A}(c)  (!_{\Sh({\cal D}, K)\slash C_{p}^{\ast}l_{J}c})\to !_{{\Sh({\cal D}', K')\slash C_{p'}^{\ast}l_{J}c}}
	\]
	is an epimorphism.
			
	\item For any $c\in {\cal C}$, the functor $A_{c}:(p\downarrow c)$ to $(p'\downarrow c)$ satisfies the first condition in the definition of morphism of sites; that is, for any object $d'$ of $\cal D$ with an arrow $\chi:p'(d')\to c$ there are a $K'$-covering family $\{g_i:d'_{i} \to d' \mid i\in I\}$ and for each $i\in I$ an arrow $\gamma_{i}:d'_{i}\to A(d_{i})$ and an arrow $u_i:p(d_{i})\to c$ such that $u_{i}\circ \phi_{d_{i}}\circ p'(\gamma_{i})=p'(g_i)$:
	\[\begin{tikzcd}
		{p'(d'_{i})} && {p'(A(d_i))} & {p(d_i)} \\
		&& c
		\arrow["{\phi_{d_{i}}}", from=1-3, to=1-4]
		\arrow["{u_i}", from=1-4, to=2-3]
		\arrow["{p'(\gamma_i)}", from=1-1, to=1-3]
		\arrow["{\chi \circ p'(g_{i})}"', from=1-1, to=2-3]
	\end{tikzcd}\]
\end{enumerate}		
\end{prop}  

\begin{proof}

(i) $\Rightarrow$ (ii) is evident, and (ii) $\Rightarrow$ (i) can be deduced from the diagrams of naturality: 

\[\begin{tikzcd}
	{\Sh (A)^* C_p^*(l_Jc)} &&& {\Sh (A)^* C_p^*(E)} \\
	\\
	\\
	{C_{p'}^*(l_Jc)} &&& {C_{p'}^*(E)}
	\arrow["{\Sh (A)^* C_p^*(a)}", from=1-1, to=1-4]
	\arrow["{\Tilde{\phi}_{E}}", from=1-4, to=4-4]
	\arrow["{C_{p'}^*(a)}"', from=4-1, to=4-4]
	\arrow["{\Tilde{\phi}_{l_Jc}}"', from=1-1, to=4-1]
\end{tikzcd}\]

Indeed, since every object $E$ is covered by the generators, and that $C_{p'}^*$ preserves the coverings, and every generator is covered by the $\Tilde{\phi}_{l_Jc}$, by transitivity we have that $\Tilde{\phi}_E$ is covering. 

(ii) $\biimp$ (iii) because the arrow in (iii) is the arrow in (ii), and so is still an epimorphism in the slice category. 

For the equivalence (iv) $\biimp$ (iii) let us recall that surjectivity in the sheaves category is given by local surjectivity in the presheaves category. The expression of $\lan_{A^{op}}C_p^*(l_Jc) \xrightarrow{\Tilde{\phi}_{l_Jc}} C_{p'}^*(l_Jc)$ in terms of a natural transformation induced between the colimits by a morphism of diagrams is indeed that every arrow $\chi: p'(d') \to c$ is locally reached, which is the point (iv).

\end{proof}

\begin{remarks}
	Condition (i) corresponds to the preservation by $\tilde{A}$, fibrewise, of the terminal object; still, it is weaker than it, as it requires the canonical arrow to merely be an epimorphism rather than an isomorphism. However, as we shall see, in presence of the other relative flatness conditions, it is enough to ensure the preservation of the terminal object (as, in fact, it is the case in the ordinary setting -- cf. the proof of Diaconescu's equivalence in \cite{maclanemoerdijk}).  
	In fact, the arrow in condition (i) or (ii) is the (unique) canonical arrow $\tilde{A}(E)({\bf !}_{\Sh({\cal D}, K)\slash C_{p}^{\ast}(E)})\to {\bf !}_{{\Sh({\cal D}', K')\slash C_{p'}^{\ast}(E)}}$ (resp. $\tilde{A}(c)({\bf !}_{\Sh({\cal D}, K)\slash C_{p}^{\ast}(l_{J}c)})\to {\bf !}_{{\Sh({\cal D}', K')\slash C_{p'}^{\ast}(l_{J}c)}}$). This is the appropriate generalization of the first condition in the notion of flat functor (or morphism of sites).
\end{remarks}

The following results provide the natural counterparts of this condition for the other conditions in the definition of morphism of sites:

\begin{prop}\label{propcond2relmorphismofsites}
Let $(A, p, p', \phi)$ be as above. Then the following conditions are equivalent:
\begin{enumerate}[(i)]
	\item For any $E\in \Sh({\cal C}, J)$ and any $U, V$ belonging to the fiber $\Sh({\cal D}, K)\slash C_{p}^{\ast}(E)$ at $E$ of $(1_{\Sh({\cal D},K)}\downarrow C_{p}^{\ast})$, the canonical arrow 
	\[
	\tilde{A}(U\times_{C_p^*E} V) \to \tilde{A}(U)\times_{C_{p'}^*E} \tilde{A}(V)
	\]
	is covering for the canonical topology of the fiber; 
	
	\item For any $c\in {\cal C}$ and any $U, V$ belonging to the fiber $\Sh({\cal D}, K)\slash C_{p}^{\ast}(l_Jc)$ at $c$ of $(1_{\Sh({\cal D},K)}\downarrow C_{p}^{\ast})$, the canonical arrow 
	\[
	\tilde{A}(U\times_{C_p^*l_Jc} V) \to \tilde{A}(U)\times_{C_{p'}^*l_Jc} \tilde{A}(V)
	\]
	is covering for the canonical topology of the fiber;
	
	\item For any $c\in {\cal C}$ and any $U, V\in \Sh({\cal D}, K)\slash C_{p}^{\ast}(l_Jc)$, the canonical arrow
	\[
	\tilde{A}(c)(U\times V) \to \tilde{A}(c)(U)\times \tilde{A}(c)(V)
	\]
	is an epimorphism.
	
	\item For any $c\in {\cal C}$, the functor $A_{c}:(p\downarrow c)$ to $(p'\downarrow c)$ satisfies the first condition in the definition of morphism of sites; that is, for any $d'\in {\cal D}'$, any arrows $h_{1}:p(d_1)\to c$ and $h_{2}:p(d_{2})\to c$ and any arrows $u:d'\to A(d_{1})$ and $v:d'\to A(d_{2})$ such that $h_{1}\circ \phi_{d_{1}}\circ p'(u)=h_{2}\circ \phi_{d_{2}}\circ p'(v)$, there are a $K'$-covering family $\{g_{i}:d'_{i}\to d' \mid i\in I\}$ and for each $i\in I$ an object $d_{i}\in {\cal D}$, an arrow $\gamma_{i}:d'_{i}\to A(d_{i})$ and arrows $s_i:d_i \to d_{1}$ and $t_i:d_i\to d_{2}$ such that $h_{1}\circ p(s_i)=h_{2}\circ p(t_{i})$ and $u\circ g_i=A(s_i)\circ \gamma_i$, $v\circ g_i=A(t_i)\circ \gamma_i$:
	
	\[\begin{tikzcd}
		{d'_i} & {d'} & {A(d_1)} && {d_i'} & {d'} & {A(d_2)} \\
		& {A(d_i)} &&&& {A(d_i)}
		\arrow["{\gamma_i}"', from=1-1, to=2-2]
		\arrow["{A(s_i)}"', from=2-2, to=1-3]
		\arrow["{g_i}", from=1-1, to=1-2]
		\arrow["u", from=1-2, to=1-3]
		\arrow["{\gamma_i}"', from=1-5, to=2-6]
		\arrow["{A(t_i)}"', from=2-6, to=1-7]
		\arrow["{g_i}", from=1-5, to=1-6]
		\arrow["v", from=1-6, to=1-7]
	\end{tikzcd}\]
\end{enumerate}	
\end{prop}

\begin{proof}
It is immediate that (i) implies (ii). In the other direction: any $C_{p'}^*E$ is covered by the $C_{p'}^*l_Jc$ and the arrows coming from them induced by the Yoneda lemma. Now, since $(1_{\Sh({\cal D},K)} \downarrow C_{p'}^*)$ is a stack, we have that an arrow being an epimorphism in the fibre over $E$ is implied by the localization of this arrows being epimorphisms in the fibre over some covering. Let us sum up the situation in diagrams: 

\[\begin{tikzcd}
	& {\Sh(A)^*U\times_{C_{p'}^*E}\Sh(A)^*V} \\
	\\
	& {\Sh(A)^*(U\times_{C_p^*E}V)} \\
	{\Sh(A)^*U} && {\Sh(A)^*V} \\
	& {C_{p'}^*E}
	\arrow[from=3-2, to=4-1]
	\arrow["{\Tilde{\phi}_E\Sh(A)^*u}", from=4-1, to=5-2]
	\arrow[from=3-2, to=4-3]
	\arrow["{\Tilde{\phi}_E\Sh(A)^*v}"', from=4-3, to=5-2]
	\arrow[dashed, from=3-2, to=1-2]
	\arrow[from=1-2, to=4-1]
	\arrow[from=1-2, to=4-3]
\end{tikzcd}\]

We first have the  arrow  mentioned in the proposition in the fiber over $E$, and then we pull back this diagram into the fiber over a representable presheaf $c$. Moreover, the pullback of a pullback being a pullback, we can rewrite $ C_{p'}^*l_Jc\times_{C_{p'}^*E}\Sh(A)^*U\times_{C_{p'}^*E}\Sh(A)^*V \simeq (\Sh(A)^*U\times_{C_{p'}^*E}C_{p'}^*l_Jc)\times_{C_{p'}^*l_Jc} (\Sh(A)^*V\times_{C_{p'}^*E}C_{p'}^*l_Jc)$

So the situation is now:

\begin{tiny}
    
\[\begin{tikzcd}[row sep=7ex, column sep=0ex]
	& {(\Sh(A)^*U\times_{C_{p'}^*E}C_{p'}^*l_Jc)\times_{C_{p'}^*l_Jc} (\Sh(A)^*V\times_{C_{p'}^*E}C_{p'}^*l_Jc)} \\
	\\
	& {C_{p'}^*l_Jc \times_{C_{p'}^*E} \Sh(A)^*(U\times_{C_p^*E}V)} \\
	{\Sh(A)^*U\times_{C_{p'}^*E}C_{p'}^*l_Jc} && {\Sh(A)^*V\times_{C_{p'}^*E}C_{p'}^*l_Jc} \\
	& {C_{p'}^*l_Jc}
	\arrow[from=1-2, to=4-1]
	\arrow[from=3-2, to=4-1]
	\arrow[from=3-2, to=4-3]
	\arrow[from=1-2, to=4-3]
	\arrow[from=4-1, to=5-2]
	\arrow[from=4-3, to=5-2]
	\arrow[dashed, from=3-2, to=1-2]
\end{tikzcd}\]
\end{tiny}




To show that this dashed arrow is an epimorphism, we make the hypothesis of (ii) appear by using the morphisms induced by $\Tilde{\phi}$: 

\begin{tiny}
    
\[\begin{tikzcd}[row sep=7ex, column sep=3ex]
	{\Sh(A)^*U\times_{\Sh(A)^*C_p^*E}\Sh(A)^*C_p^*l_Jc} && {\Sh(A)^*C_p^*l_Jc \times_{\Sh(A)^*C_p^*E} \Sh(A)^*(U\times_{C_p^*E}V)} && {\Sh(A)^*V\times_{\Sh(A)^*C_p^*E}\Sh(A)^*C_p^*l_Jc} \\
	\\
	{\Sh(A)^*U\times_{C_{p'}^*E}C_{p'}^*l_Jc} && {C_{p'}^*l_Jc \times_{C_{p'}^*E} \Sh(A)^*(U\times_{C_p^*E}V)} && {\Sh(A)^*V\times_{C_{p'}^*E}C_{p'}^*l_Jc}
	\arrow["{<1_{\Sh(A)^*U},\Tilde{\phi_c}>}", dotted, from=1-1, to=3-1]
	\arrow["{<1_{\Sh(A)^*(U\times_{C_p^*}V)},\Tilde{\phi_c}>}"', dotted, from=1-3, to=3-3]
	\arrow["{<1_{\Sh(A)^*V},\Tilde{\phi_c}>}", dotted, from=1-5, to=3-5]
	\arrow[from=1-3, to=1-1]
	\arrow[from=1-3, to=1-5]
	\arrow[from=3-3, to=3-1]
	\arrow[from=3-3, to=3-5]
\end{tikzcd}\]
\end{tiny}

Since $\Sh(A)^*$ commutes with finite limits, we have: \\ 

$$\Sh(A)^*C_p^*l_Jc \times_{\Sh(A)^*C_p^*E} \Sh(A)^*(U\times_{C_p^*E}V)$$ 
$$ \simeq $$ 
$$ \Sh(A)^*(C_p^*l_Jc\times_{C_p^*E}U\times_{C_p^*E}V)$$ and, as the pullback of a pullback is still a pullback: \\

$$\Sh(A)^*(C_p^*l_Jc\times_{C_p^*E}U\times_{C_p^*E}V) $$ 
$$\simeq$$
$$\Sh(A)^*((C_p^*l_Jc \times_{C_p^*l_Jc} U) \times_{C_{p'}^*l_Jc} \Sh(A)^*(C_p^*l_Jc \times_{C_p^*l_Jc} V)) $$ \\
Moreover, by the universal property of: 

$(\Sh(A)^*U\times_{C_{p'}^*E}C_{p'}^*l_Jc)\times_{C_{p'}^*l_Jc} (\Sh(A)^*V\times_{C_{p'}^*E}C_{p'}^*l_Jc)$,  we now have the dashed arrow:

\begin{tiny}
    
\[\begin{tikzcd}[row sep=7ex, column sep=0ex]
	& {\Sh(A)^*((C_p^*l_Jc \times_{C_p^*l_Jc} U) \times_{C_{p'}^*l_Jc} \Sh(A)^*(C_p^*l_Jc \times_{C_p^*l_Jc} V)) } \\
	\\
	& { \Sh(A)^*((C_p^*l_Jc \times_{C_p^*l_Jc} U) \times_{C_p^*l_Jc} (C_p^*l_Jc \times_{C_p^*l_Jc} V)) } \\
	\\
	{\Sh(A)^*U\times_{C_{p'}^*E}C_{p'}^*l_Jc} && {\Sh(A)^*V\times_{C_{p'}^*E}C_{p'}^*l_Jc} \\
	& {C_{p'}^*l_Jc}
	\arrow["{\Tilde{\phi}_c\Sh(A)^*\pi}", from=3-2, to=5-1]
	\arrow["{\Tilde{\phi}_c\Sh(A)^*\pi'}"', from=3-2, to=5-3]
	\arrow[from=1-2, to=5-1]
	\arrow[from=1-2, to=5-3]
	\arrow[from=5-1, to=6-2]
	\arrow[from=5-3, to=6-2]
	\arrow[dashed, from=3-2, to=1-2]
\end{tikzcd}\]
\end{tiny}

This arrow is, by assumption, an epimorphism. But, by what precedes we have the following commutative diagram:

\begin{tiny}
    \[\begin{tikzcd}[row sep=7ex, column sep=0ex]
	{\Sh(A)^*((C_p^*l_Jc \times_{C_p^*l_Jc} U) \times_{C_{p'}^*l_Jc} \Sh(A)^*(C_p^*l_Jc \times_{C_p^*l_Jc} V)) } \\
	&& {C_{p'}^*l_Jc \times_{C_{p'}^*E} \Sh(A)^*(U\times_{C_p^*E}V)} \\
	{ \Sh(A)^*((C_p^*l_Jc \times_{C_p^*l_Jc} U) \times_{C_p^*l_Jc} (C_p^*l_Jc \times_{C_p^*l_Jc} V)) }
	\arrow[dashed, from=3-1, to=1-1]
	\arrow[dashed, from=2-3, to=1-1]
	\arrow["{<\Tilde{\phi}_c,1>}"', from=3-1, to=2-3]\end{tikzcd}\]

\end{tiny}

That gives us that the localizations factorize through epimorphisms, and are therefore epimorphisms. Hence, our initial arrow is an epimorphism. \\

The equivalence between (iii) and (ii) is because the product in the fibres are the pullbacks in the topos, and an morphism is an epimorphism in a fiber if and only if it is already one in the topos (one hand is obvious, and the other one comes from the fact that postomposition with the unique arrow to the terminal of the basis is a left adjoint to pullback, and the fiber over the terminal is the topos itself). \\

The equivalence between (iv) and (ii) is due to the fact that we can restrict on the generators not only on the "target" of the pullback, but also on the objects at the source of the legs, since $\Tilde{A}$ preserves colimits ($\Sh(A)^*$ does and the post-composition with an arrow does too in the slice categories), as well as pullbacks. The epimorphism condition is hence equivalent to the pointwise locally epimorphism condition on the representable presheaves, of which (iv) is a reformulation: the first part says that the two arrows lie in the pullback, and the second that there exists a covering on which the localizations of these arrows are in the image of the morphism.

\end{proof}

Let us now turn to the third condition in the definition of morphism of sites:

\begin{prop}\label{propcond3relmorphismofsites}
	Let $(A, p, p', \phi)$ be as above. Then the following conditions are equivalent:
	\begin{enumerate}[(i)]
		\item For any $E\in \Sh({\cal C}, J)$ and any pair of parallel arrows $h_{1}, h_{2}:U\to V$ belonging to the fibre $\Sh({\cal D}, K)\slash C_{p}^{\ast}(E)$ at $E$ of $(1_{\Sh({\cal D},K)}\downarrow C_{p}^{\ast})$, the canonical arrow 
		\[
		\tilde{A}(\textup{Eq}(h_{1}, h_{2})) \to \textup{Eq}(\tilde{A}(h_{1}), \tilde{A}(h_{2}))
		\]
		is covering for the canonical topology of the fiber; 
		
		\item For any $c\in {\cal C}$ and any pair of parallel arrows $h_{1}, h_{2}:U\to V$ belonging to the fibre $\Sh({\cal D}, K)\slash C_{p}^{\ast}(l_Jc)$ at $c$ of $(1_{\Sh({\cal D},K)}\downarrow C_{p}^{\ast})$, the canonical arrow 
		\[
		\tilde{A}(\textup{Eq}(h_{1}, h_{2})) \to \textup{Eq}(\tilde{A}(h_{1}), \tilde{A}(h_{2}))
		\]
		is covering for the canonical topology of the fiber; 
		
		\item For any $c\in {\cal C}$ any pair of parallel arrows $h_{1}, h_{2}:U\to V$ in $\Sh({\cal D}, K)\slash C_{p}^{\ast}(l_Jc)$, the canonical arrow 
		\[
		\tilde{A}(c)(\textup{Eq}(h_{1}, h_{2})) \to \textup{Eq}(\tilde{A}(c)(h_{1}), \tilde{A}(c)(h_{2}))
		\]
		is an epimorphism.
		
		\item For any $c\in {\cal C}$, the functor $A_{c}$ satisfies the third condition in the definition of morphism of sites.
		
		\item The functor $A$ satisfies the third condition in the definition of morphism of sites; that is, for any pair of arrows $f_1 , f_2: d_1 \rightrightarrows d_2$ of ${\mathcal D}$ and any arrow of ${\mathcal D}'$
		$$
		g: d' \longrightarrow A(d_{1})
		$$
		satisfying
		$$
		A(f_1) \circ g = A(f_2) \circ g \, ,
		$$
		there exist a $K'$-covering family
		$$
		g_i: d_i' \longrightarrow d' \, , \quad i \in I \, ,
		$$
		and a family of morphisms of ${\mathcal D}$
		$$
		k_i: d_i \longrightarrow d_{1} \, , \quad i \in I \, ,
		$$
		satisfying
		$$
		f_1 \circ k_i = f_2 \circ k_i \, , \quad \forall \, i \in I \, ,
		$$
		and of morphisms of ${\mathcal D}'$
		$$
		\gamma_i: d_i' \longrightarrow F(d_i) \, , \quad i \in I \, ,
		$$
		making the following square commutative:
		$$
		\xymatrix{
			d_i' \ar[r]^{g_i} \ar[d]_{\gamma_i} & d' \ar[d]^{g} \\
			A(d_i) \ar[r]^{F(k_i)} & A(d_{1})
		} 
		$$ 
	\end{enumerate}
\end{prop}

\begin{proof}
As for the precedent, the equivalence (ii) $\Leftrightarrow$ (i) can be proven by pulling back along the generators, saying that pullbacks commute with equalizers, and using $\Tilde{\phi}$ to have a factorisation of the induced arrow by the arrow being stated as epimorphism in (ii), concluding that this is itself an epimorphism.

The equivalence (ii) $\Leftrightarrow$ (iii) is due to the fact that equilazers in the slice toposes are the same as equalizers in the absolute topos, as well as the epimorphisms.

The equivalence  (ii) $\Leftrightarrow$ (iv) is again the reformulation on the representable presheaves of the locally epimorphic condition on the presheaves for this morphism to be an epimorphism in the category of sheaves. 

And last two points are equivalent since we don't have to stipulate an indexing arrow of the form $p(d') \to c$, as it is directly given as a composite.

\end{proof}

Lastly, let us analyze the relative analogue of the cover-preservation condition in the notion of morphism of sites:

\begin{prop}\label{propcond4relmorphismofsites}
	Let $(A, p, p', \phi)$ be as above. Then the following conditions are equivalent:
\begin{enumerate}[(i)]
	\item For any $E\in \Sh({\cal C}, J)$ and any family of covering arrows in the fibre $\Sh({\cal D}, K)\slash C_{p}^{\ast}(E)$ at $E$ of $(1_{\Sh({\cal D},K)} \downarrow C_{p}^{\ast})$, its image under $\tilde{A}$ is a  covering family.
	
	\item For any $c\in {\cal C}$ and any family of covering arrows in the fibre $\Sh({\cal D}, K)\slash C_{p}^{\ast}(l_Jc)$ at $l_Jc$ of $(1_{\Sh({\cal D},K)} \downarrow C_{p}^{\ast})$, its image under $\tilde{A}$ is a covering family.
	
	\item  For any $c\in {\cal C}$, the functor
	\[
	\tilde{A}(c):\Sh({\cal D}, K)\slash C_{p}^{\ast}(l_{J}c) \to \Sh({\cal D}', K')\slash C_{p'}^{\ast}(l_{J}c) 
	\]
	preserves epimorphic families. 
	
	\item For any $c\in {\cal C}$, the functor $A_{c}$ is cover-preserving. 
	
	\item The functor $A$ is cover-preserving.
\end{enumerate}	
\end{prop}

\begin{proof}

Let us recall that in a topos, a family of arrows is covering if and only if the arrow coming from its coproduct is epimorphic. To have the implication (ii) $\Rightarrow$ (i) (the other direction being immediate), we use the fact that an arrow is an epimorphism in the fiber over $E$ if and only if its localizations are already ones in every of the fibers over the representable presheaves (covering $E$ by the co-Yoneda lemma). These localizations are epimorphisms by the same argument as for the precedent propositions (using that $\Sh(A)^*$ commutes with pullbacks and coproducts, and factorizing by the coproduct in the fibers over representables presheaves, which is epimorphic by (ii)).

The point (iii) is just a reformulation of (ii), taking in account that the covering families in the fiber are also the epimorphic ones.

We have that (v) implies (iv) because $A_c$ sends an arrow $f$ to $A(f)$, and the covering sieves these slice categories are those who are already covering in $\cal D$ resp. $\cal D'$. And (iv) implies (v) because, if we have a covering sieve $S$ over $d$, we can look at this covering sieve over $p(d) \xrightarrow{1_{p(d)}} p(d)$ in $p \downarrow p(d)$, which is covering by definition, and sent into  $p' \downarrow p(d)$ on a covering family by $A_{p(d)}$, which says exactly that $A(S)$ is covering.

Then, if $\Tilde{A}$ is cover preserving, we have that, for every covering sieve $S$ over some $d$, $\Tilde{A}$ sends the sieve $l_K(S)$ over the terminal object in the fiber over $l_K(d)$ to an epimorphic family. This family is, by definition of $\Tilde{A}$, given by $l_{J'}A(S)$, which is epimorphic if and only if $A(S)$ is covering. This gives (ii) $\Rightarrow$ (v). Finally (v) $\Rightarrow$ (ii) because we can pull back every covering sieve of any object along the generators, and then locally represent the arrows obtained as coming from the site.  

\end{proof}

This thus leads us to the following definition/characterization of lax morphisms of sites over $({\cal C},J)$:

\begin{prop}\label{propmorphimrelmorphism}
	Let $p:({{\cal D}, K})\to ({\cal C}, J)$ and $p':({\cal D}', K')\to ({\cal C}, J)$ be comorphisms of sites and $A$ a functor ${\cal D}\to {\cal D}'$ together with a natural transformation $\phi: p'\circ A\Rightarrow p$ (N.B. We do not suppose that $A$ is a morphism of ordinary sites). Then the following conditions are equivalent:
	\begin{enumerate}
		\item The functor $A$ together with $\phi$ constitute a lax morphism of sites over $({\cal C},J)$, that is, it (is a morphism of ordinary sites which) induces a geometric morphism $\Sh(A)$ over $\Sh({\cal C}, J)$;
		
		\item For any $c\in {\cal C}$, $A_{c}$ is a morphism of sites $((p\downarrow c), K_c)\to ((p'\downarrow c), K'_{c})$; 
		
		\item The functor $A$ together with $\phi$ satisfy the following 'relative local filteredness' and cover preserving conditions: 
		\begin{enumerate}
			\item (condition (iv) of Proposition \ref{propcond1relmorphismofsites})  For any $c\in {\cal C}$ and any object $d'$ of $\cal D$ with an arrow $\chi:p'(d')\to c$, there are a $K'$-covering family $\{g_i:d'_{i} \to d' \mid i\in I\}$ and for each $i\in I$ an arrow $\gamma_{i}:d'_{i}\to A(d_{i})$ and an arrow $u_i:p(d_{i})\to c$ such that $u_{i}\circ \phi_{d_{i}}\circ p'(\gamma_{i})=p'(g_i)\circ \chi$:
			\[\begin{tikzcd}
				{p'(d'_{i})} && {p'(A(d_i))} & {p(d_i)} \\
				&& c
				\arrow["{\phi_{d_{i}}}", from=1-3, to=1-4]
				\arrow["{u_i}", from=1-4, to=2-3]
				\arrow["{p'(\gamma_i)}", from=1-1, to=1-3]
				\arrow["{\chi \circ p'(g_{i})}"', from=1-1, to=2-3]
			\end{tikzcd}\]
			
			\item (condition (iv) of Proposition \ref{propcond2relmorphismofsites}) For any $c\in {\cal C}$, any $d'\in {\cal D}'$, any arrows $h_{1}:p(d_1)\to c$ and $h_{2}:p(d_{2})\to c$ and any arrows $u:d'\to A(d_{1})$ and $v:d'\to A(d_{2})$ such that $h_{1}\circ \phi_{d_{1}}\circ p'(u)=h_{2}\circ \phi_{d_{2}}\circ p'(v)$, there are a $K'$-covering family $\{g_{i}:d'_{i}\to d' \mid i\in I\}$ and for each $i\in I$ an object $d_{i}\in {\cal D}$, an arrow $\gamma_{i}:d'_{i}\to A(d_{i})$ and arrows $s_i:d_i \to d_{1}$ and $t_i:d_i\to d_{2}$ such that $h_{1}\circ p(s_i)=h_{2}\circ p(t_{i})$ and $u\circ g_i=A(s_i)\circ \gamma_i$, $v\circ g_i=A(t_i)\circ \gamma_i$:
			
			\[\begin{tikzcd}
				{d'_i} & {d'} & {A(d_1)} && {d_i'} & {d'} & {A(d_2)} \\
				& {A(d_i)} &&&& {A(d_i)}
				\arrow["{\gamma_i}"', from=1-1, to=2-2]
				\arrow["{A(s_i)}"', from=2-2, to=1-3]
				\arrow["{g_i}", from=1-1, to=1-2]
				\arrow["u", from=1-2, to=1-3]
				\arrow["{\gamma_i}"', from=1-5, to=2-6]
				\arrow["{A(t_i)}"', from=2-6, to=1-7]
				\arrow["{g_i}", from=1-5, to=1-6]
				\arrow["v", from=1-6, to=1-7]
			\end{tikzcd}\]
			
			\item  (condition (v) of Proposition \ref{propcond3relmorphismofsites})  For any pair of arrows $f_1 , f_2: d_1 \rightrightarrows d_2$ of ${\mathcal D}$ and any arrow of ${\mathcal D}'$
			$$
			g: d' \longrightarrow A(d_{1})
			$$
			satisfying
			$$
			A(f_1) \circ g = A(f_2) \circ g \, ,
			$$
			there exist a $K'$-covering family
			$$
			g_i: d_i' \longrightarrow d' \, , \quad i \in I \, ,
			$$
			and a family of morphisms of ${\mathcal D}$
			$$
			k_i: d_i \longrightarrow d_{1} \, , \quad i \in I \, ,
			$$
			satisfying
			$$
			f_1 \circ k_i = f_2 \circ k_i \, , \quad \forall \, i \in I \, ,
			$$
			and of morphisms of ${\mathcal D}'$
			$$
			\gamma_i: d_i' \longrightarrow F(d_i) \, , \quad i \in I \, ,
			$$
			making the following square commutative:
			$$
			\xymatrix{
				d_i' \ar[r]^{g_i} \ar[d]_{\gamma_i} & d' \ar[d]^{g} \\
				A(d_i) \ar[r]^{A(k_i)} & A(d_{1})
			} 
			$$ 
			\item (condition (v) of Proposition \ref{propcond4relmorphismofsites}) The functor $A$ is cover-preserving.
		\end{enumerate}		
	\end{enumerate}
\end{prop}

\begin{proof}
    Points 2 and 3 are equivalent by the description of morphisms of sites as locally filtering cover-preserving functors. 
    
    If $(A,\phi)$ is a morphism of sites over $({\cal C},J)$, then the $\Tilde{A}(c)$ are geometric morphisms because of the charaterization in \ref{etaextensionproperties}. They are induced by the $A_c$ as explicited in \ref{morphsitesinducing}, which are hence morphisms of sites, \textit{i.e.} satisfying the equivalent points 2 and 3. 
    
    If the $A_c$ are morphisms of sites, they induce the geometric morphisms $\Tilde{A}(c)$, which respect finite limits, and again because of the characterization of \ref{etaextensionproperties}, $A$ is a morphism of sites over $({\cal C},J)$.
\end{proof}

\begin{remark}\label{remrelativemorphismimpliesordinarymorphism}
	The conditions of Proposition \ref{propmorphimrelmorphism}, for an arbitrary functor $A$ together with a natural transformation $\phi:p'\circ A \Rightarrow p$, actually imply that $A$ is a morphism of sites. Indeed, if the $A_{c}$ are  morphisms of sites for every $c\in {\cal C}$ then we have geometric morphisms
	\[
	\Sh(A_{c}):\Sh({\cal D}', K')\slash C_{p'}^{\ast}(l_Jc) \to \Sh({\cal D}, K)\slash C_{p}^{\ast}(l_Jc)
	\]
     which are functorially compatible with each other. By using the glueing property of the canonical stack of a topos, we can then glue these geometric morphisms to a geomeric morphism of the slice topos over the colimit of all the $C_p^*(l_Jc)$, that is the terminal object. Thus, we glue them into a geometric morphism 
	\[
	\Sh({\cal D}', K') \to \Sh({\cal D}, K)
	\]
	giving that $A$ is a morphism of sites (cf. Proposition B3.1.5 \cite{elephant}, which gives an equivalence between the category of relative geometric morphisms and that of indexed geometric morphisms between the slice toposes).
\end{remark}

\subsubsection{Formulation for relative sites}

The following result motivates our notion of morphism of relative sites, and, in particular, that of flat functor relative to a base site, appearing in our generalisation of Diaconescu's theorem in the relative setting (Theorem \ref{diaconescufibration}).

\begin{prop}\label{propmorphismfibrationsrelmorphismsites}
In the same setting as Proposition \ref{etaextensionproperties}, if $A$ is a morphism of fibrations then $\tilde{A}$ is a morphism of fibrations.
    In other words, every morphism of relative sites (in the sense of Definition \ref{defmorphismrelsites}(c)) induces a morphism of relative toposes.
\end{prop}

\begin{proof}
Let us express $p$ and $p'$ as the fibrations respectively associated with indexed categories ${\mathbb D}$ and ${\mathbb D}'$. Slightly abusing notation, we shall denote by $A_{c}$ (for an object $c$ of $\cal C$) the functor ${\mathbb D}(c) \to {\mathbb D}'(c)$ induced by the morphism of fibrations $A$.
 We can clearly suppose, at the cost of taking left Kan extensions along the Yoneda embedding of $\cal C$ into $\hat{\cal C}$ (note that this operation sends morphisms of fibrations to morphisms of fibrations -- cf. section 3.2 of \cite{CaramelloZanfa} for Kan extensions of fibrations), that the category $\cal C$ has finite limits. 
    
    Under this assumption, we can show that, for any arrow $f:c\to c'$ in $\cal C$, we have a commutative diagram
\[\begin{tikzcd}
	{(p \downarrow c)} && {(p'\downarrow c)} \\
	{(p\downarrow c')} && {(p'\downarrow c')}
	\arrow["{A_{c'}}", from=2-1, to=2-3]
	\arrow["{A_{c}}", from=1-1, to=1-3]
	\arrow["{\chi^{p}_{f}}", from=2-1, to=1-1]
	\arrow["{\chi^{p'}_{f}}", from=2-3, to=1-3]
\end{tikzcd}\] 
    of morphisms of sites, inducing the commutative diagram

\[\begin{tikzcd}
	{\textup{\bf Sh}({\cal D}, K) {\slash} C_{p}^{\ast}(l_{J}(c))  } && {\Sh({\cal D}', K')\slash C_{p'}^{\ast}(l_{J}(c))} \\
	{\textup{\bf Sh}({\cal D}, K) {\slash} C_{p}^{\ast}(l_{J}(c'))} && {\Sh({\cal D}', K')\slash C_{p'}^{\ast}(l_{J}(c'))}
	\arrow["{\tilde{A}_{c}}", from=1-1, to=1-3]
	\arrow["{\tilde{A}_{c'}}", from=2-1, to=2-3]
	\arrow["{C_{p}^{\ast}(l_{J}(f))}", from=2-1, to=1-1]
	\arrow["{C_{p'}^{\ast}(l_{J}(f))}", from=2-3, to=1-3]
\end{tikzcd}\]
at the level of the associated toposes.

Indeed, the vertical morphisms $\chi^{p}_{f}$ and $\chi^{p'}_{f}$ in the first diagram above are defined as follows. They send an object $((c'', x), g:c''=p((c'', x)) \to c')$ (resp.  $((c'', x), g:c''=p'((c'', x)) \to c')$) of $(p\downarrow c')$ (resp. of $(p'\downarrow c')$) to the object $(e, {\mathbb D}(k)(x), h:e\to c)$ (resp. $(e, {\mathbb D}'(k)(x), h:e\to c)$) of $(p\downarrow c)$ (resp. of $(p'\downarrow c)$), where the following is a pullback square in $\cal C$:

\[\begin{tikzcd}
	e & c \\
	{c''} & {c'}
	\arrow["f", from=1-2, to=2-2]
	\arrow["g"', from=2-1, to=2-2]
	\arrow["k"', from=1-1, to=2-1]
	\arrow["h", from=1-1, to=1-2]
\end{tikzcd}\]

The above diagram commutes precisely since $A$ is a morphism of fibration, equivalentely, of indexed categories ${\mathbb D}\to {\mathbb D}'$ over $\cal C$ (this conditions entails that ${\mathbb D}'(k)((A_{c''}(x))\cong A_{e}({\mathbb D}(k)(x))$).

We can give an alternative proof by means of the cofinality conditions of Theorem \ref{thmcofinality}:

First, let us show that surjectivity requirement. Let $(x',c')$ be an object of the fibration $\cal D'$ and $u : c \to \overline{c}$. By the filtering characterization of $A$, we have a set $I$ together with a covering family $(u_i,g_i) 
: (x_i',c_i') \to (x,c)$ and objects $(x_i,c_i)$ of $\cal D$ with arrows $(v_i,h_i) : (x_i',c_i') \to A(x_i,c_i) $. To fulfill all the conditions of surjectivity, there remains to obtain an arrow indexing $A(x_i,c_i)$. In fact, since $A$ is a morphism of fibrations, we can use the action of $h_i$ through $\cal D$ seen as an indexed category $\mathbb D$, and, by the naturality of $A$ seen as an indexed functor, we are still working with an object coming from the functor:

\[\begin{tikzcd}
	{(x',c')} \\
	{(x_i',c_i')} && {A(x_i,c_i)} \\
	& {A({\mathbb D}(f)(x_i),c_i')}
	\arrow["{(u_i,g_i)}", from=2-1, to=1-1]
	\arrow["{(v_i,h_i)}"', from=2-1, to=2-3]
	\arrow["{(v_i,1)}"', from=2-1, to=3-2]
	\arrow["{(1,h_i)}"', from=3-2, to=2-3]
\end{tikzcd}\]

Where the bottom arrow is factored as a vertical and an horizontal arrow. Since we are reduced to work in a generalized element of $A$ in the fiber over $c_i$, we have now an obvious indexing arrow for it: $u \circ g_i$.

To show that the injectivity requirement also holds, let us take two generalized elements of $A$ indexed by some arrows $u_1 : c_1 \to \overline{c}$ and $u_2 : c_2 \to \overline{c}$, as follows:
\[\begin{tikzcd}
	& {A(x_1,c_1)} && {c_1} \\
	{(x',c')} && {c'} && {\overline{c}} \\
	& {A(x_2,c_2)} && {c_2}
	\arrow["{(v_1,f_1)}", from=2-1, to=1-2]
	\arrow["{(v_2,f_2)}"', from=2-1, to=3-2]
	\arrow["{u_1}", from=1-4, to=2-5]
	\arrow["{u_2}"', from=3-4, to=2-5]
	\arrow["{f_2}"', from=2-3, to=3-4]
	\arrow["{f_1}", from=2-3, to=1-4]
\end{tikzcd}\]

As hinted in the Remark \ref{remcof}(b), since $A$ is a morphism of sites, we may want to join locally these two indexed generalized elements by diagrams product coming from indexed generalized elements of $A$. By the filtering characterization of $A$ we obtain a set $I$ with a covering family $(w_i,g_i)$ such that we have the situation as in the following commutative square:

\[\begin{tikzcd}
	{(x',c')} & {A(x_1,c_1)} \\
	{(x_i',c_i')} & {A(x_i,c_i)}
	\arrow["{(v_1,f_1)}", from=1-1, to=1-2]
	\arrow["{(v_i,f_i)}"', from=2-1, to=2-2]
	\arrow["{(w_i,g_i)}", from=2-1, to=1-1]
	\arrow["{A(w_i^1,g_i^1)}"', from=2-2, to=1-2]
\end{tikzcd}\]

and same with $2$ instead of $1$. As for the surjectivity, we lack of indexing arrows for our generalized elements $(v_i,f_i)$. But, again, we can bring the $(x_i,c_i)$ into the fibers over $c_i'$ by the action of $\cal D$ as an indexed functor:

\[\begin{tikzcd}
	{(x',c')} & {A(x_1,c_1)} \\
	{(x_i',c_i')} & {A({\mathbb D}(f_i)(x_i),c_i')}
	\arrow["{(v_1,f_1)}", from=1-1, to=1-2]
	\arrow["{(v_i,1)}", from=2-1, to=2-2]
	\arrow["{(w_i,g_i)}", from=2-1, to=1-1]
	\arrow["{A(\mathbb D(f_i)(w_i^1),g_i^1f_i)}"', from=2-2, to=1-2]
\end{tikzcd}\]

where we used the factorization of the bottom arrow as an horizontal and a vertical arrow, together with the fact that $A$ sends cartesian arrows to cartesian arrows. Also, we have the same commutative square for $1$ replaced by $2$. Now, since we are in the fiber over $c_i'$, we have an obvious indexing arrow for $A({\mathbb D}(f_i)(x_i),c_i')$, given by: $u_1\circ f_1 \circ g_i$.

\end{proof}

\begin{defn}
	Let $f:{\cal F}\to \Sh({\cal C}, J)$ be a $\Sh({\cal C}, J)$-topos and ${\mathbb D}$ a $\cal C$-indexed category. We call a $\cal C$-indexed functor $A:\mathcal{G}({\mathbb D})\to (1_{\cal F} \downarrow f^{\ast} l_J)$ \emph{flat relative to $\Sh({\cal C}, J)$} when it yields a morphism of relative sites $({\cal G}({\mathbb D}), J_{\mathbb D})\to ((1_{\cal F} \downarrow f^{\ast} l_J), J_f)$.
	
	If $K$ is a Grothendieck topology on ${\cal G}({\mathbb D})$ containing $J_{\mathbb D}$ then a flat relative to $\Sh({\cal C}, J)$ functor $A$ is said to be \emph{$K$-continuous} if $A$ yields a morphism of relative sites $({\cal G}({\mathbb D}), K)\to ((1_{\cal F} \downarrow f^{\ast}l_J), J_f)$.    
\end{defn}

We notice that the cover-preservation condition, in the particular case of the Giraud topology, is automatically satisfied, since, as shown in \cite{denseness} Cor. 4.47., every morphism of fibrations is continuous with respect to the Giraud topologies. 

By the universal property of the comma category $(1_{\cal F} \downarrow f^{\ast} l_J)$, as in \ref{etaextensionproperties} (v), giving a functor $A:{\cal G}({\mathbb D})\to (1_{\cal F} \downarrow f^{\ast} l_J)$ amounts to give a functor $A_{\cal F}:{\cal G}({\mathbb D})\to {\cal F}$ together with a natural transformation $s:A_{\cal F}\to f^{\ast} l_J  p_{\mathbb D}$; for any $(x, c)\in {\cal G}({\mathbb D})$ we shall write $s_{x}:A_{\cal F}(x)\to f^{\ast}l_J(c)$. Here, we sum-up the description of a relative locally filtering functor between a fibration and the canonical relative site of a relative topos:
	
\begin{prop}\label{filtrantavaleurtopos}
	Let $f:{\cal F}\to \Sh({\cal C}, J)$ be a $\Sh({\cal C}, J)$-topos, $\mathcal{G}(\mathbb D)$ a $\cal C$-indexed category and $A:{\mathcal{G}(\mathbb D})\to (1_{\cal F} \downarrow f^{\ast} l_J)$ a morphism of fibrations. Then is flat relative to $\Sh({\cal C}, J)$ if and only if the following conditions are satisfied:
	\begin{enumerate}[(i)]
		\item For any generalized element $f:F\to f^{\ast}l_J(c)$, there are an epimorphic family $g_i : F_i \to F$	and, for each $i\in I$, a generalized element $\gamma_{i}:F_i\to A(x_{i}, c_i)$ and an arrow $u_i:c_i\to c$ such that $f^{\ast}l_J(u_i)\circ s_{x_{i}}\circ \gamma_{i}=g_i$:
	      
	    \[\begin{tikzcd}
	    	{F_i} & F & {f^{\ast}l_J(c)} \\
	    	& {A(x_i, c_i)} & {f^{\ast}l_J(c_i)}
	    	\arrow["f", from=1-2, to=1-3]
	    	\arrow["{g_i}", from=1-1, to=1-2]
	    	\arrow["{f^{\ast}l_J(u_i)}"', from=2-3, to=1-3]
	    	\arrow["{s_{x_i}}"', from=2-2, to=2-3]
	    	\arrow["{\gamma_i}"', from=1-1, to=2-2]
	    \end{tikzcd}\]

		\item For any $c\in {\cal C}$ and arrows $h_{1}: c_1 \to c$, $h_{2}:c_2\to c$, for any generalized element $a:F\to f^{\ast}l_J(c')$ and generalized elements $u:F\to A_{\cal F}(x_{1}, c_1)$ and $v:F\to A_{\cal F}(x_{2}, c_2)$ and arrows $a_{1}:c'\to c_1$ and $a_{2}:c' \to c_2$ such that $s_{x_{1}}\circ u=f^{\ast}l_J(a_1)\circ a$, $s_{x_{2}}\circ v=f^{\ast}l_J(a_2)\circ a$ and $h_1\circ a_{1}=h_2\circ a_{2}$,
		
		
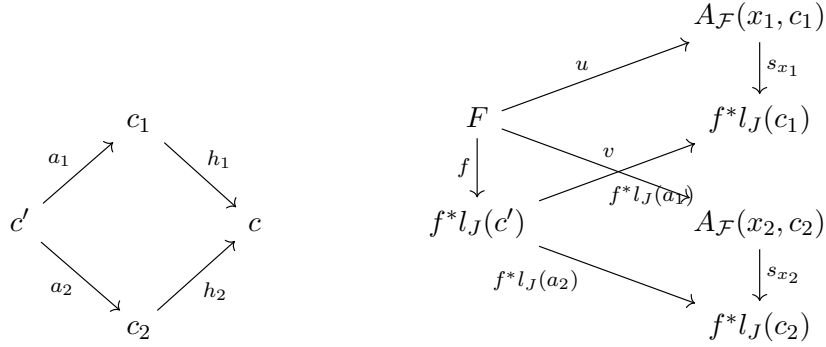
\[\begin{tikzcd}
			&&&&&& {A_{\cal F}(x_1, c_1)} \\
			& {c_1} &&& F && {f^{\ast}l_J(c_1)} \\
			{c'} && c && {f^{\ast}l_J(c')} && {A_{\cal F}(x_2, c_2)} \\
			& {c_2} &&&&& {f^{\ast}l_J(c_2)}
			\arrow["{h_1}", from=2-2, to=3-3]
			\arrow["{h_2}"', from=4-2, to=3-3]
			\arrow["{a_2}"', from=3-1, to=4-2]
			\arrow["{a_1}", from=3-1, to=2-2]
			\arrow["f"', from=2-5, to=3-5]
			\arrow["v", from=2-5, to=3-7]
			\arrow["{s_{x_2}}", from=3-7, to=4-7]
			\arrow["{f^{\ast}l_J(a_2)}"'{pos=0.3}, from=3-5, to=4-7]
			\arrow["{s_{x_1}}", from=1-7, to=2-7]
			\arrow["u", from=2-5, to=1-7]
			\arrow["{f^{\ast}l_J(a_1)}"'{pos=0.4}, from=3-5, to=2-7]
		\end{tikzcd}\]
		
		there exists a set $I$ such that we have an epimorphic family $g_i: F_i \to F$
		and, for each $i\in I$, a generalized element $\gamma_{i}:F_i\to A_{\cal F}(x_{i}, c_i)$, arrows $f_{i}:F_{i}\to f^{\ast}l_J(c'_i)$ in $\cal F$, $w_i:c_i'\to c'$ in $\cal C$, $z_i:c_i' \to c_i$ in $\cal C$, $(p_i, b_i^1):(x_{i}, c_i) \to (x_1, c_1)$ and $(q_i, b_i^2):(x_{i}, c_i) \to (x_2, c_2)$ in ${\cal G}({\mathbb D})$ such that $f\circ g_i=f^{\ast}l_J(w_i)\circ f_i$, $s_{x_i}\circ \gamma_i=f^{\ast}l_J(z_i)\circ f_i$, $h_1\circ b_i^1=h_2 \circ b_i^2$, $u\circ g_i=A_{\cal F}(p_i, b_i^1)\circ \gamma_i$, and $v\circ g_i=A_{\cal F}(q_i, b_i^2)\circ \gamma_i$:     
		
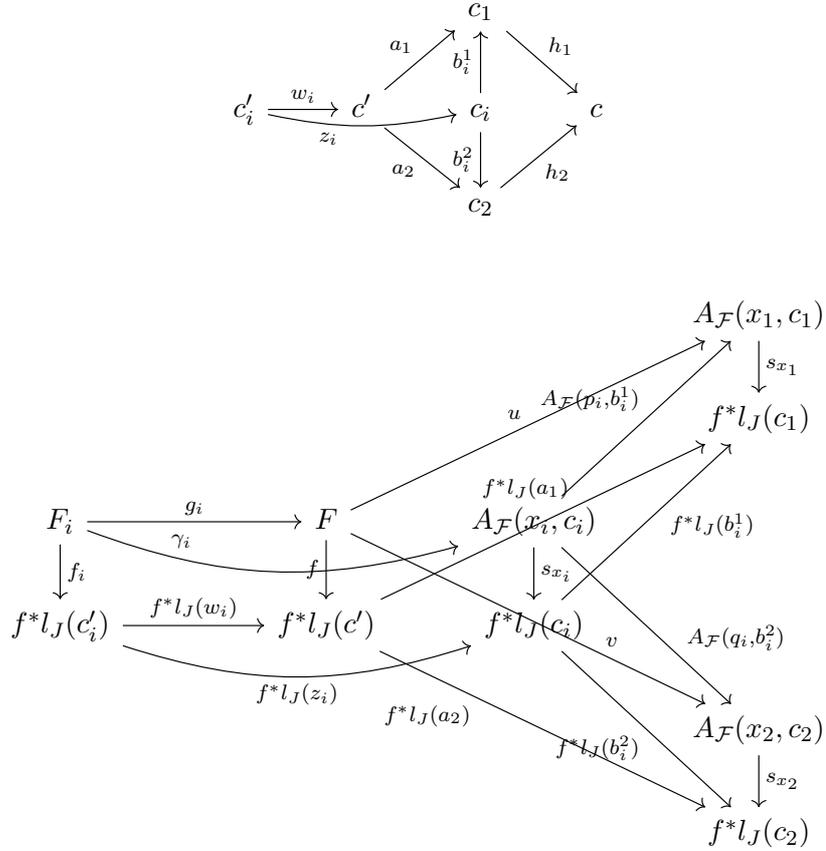
		\[\begin{tikzcd}
			&& {c_1} \\
			{c_i'} & {c'} & {c_i} & c \\
			&& {c_2}
			\arrow["{h_1}", from=1-3, to=2-4]
			\arrow["{h_2}"', from=3-3, to=2-4]
			\arrow["{a_2}"', from=2-2, to=3-3]
			\arrow["{a_1}", from=2-2, to=1-3]
			\arrow["{w_i}", from=2-1, to=2-2]
			\arrow["{b_i^1}", from=2-3, to=1-3]
			\arrow["{b_i^2}"', from=2-3, to=3-3]
			\arrow["{z_i}"'{pos=0.4}, bend right=12, from=2-1, to=2-3]
		\end{tikzcd}\]

		\[\begin{tikzcd}
			&&&& {A_{\cal F}(x_1, c_1)} \\
			&&&& {f^{\ast}l_J(c_1)} \\
			{F_i} && F & {A_{\cal F}(x_i, c_i)} \\
			{f^{\ast}l_J(c'_i)} && {f^{\ast}l_J(c')} & {f^{\ast}l_J(c_i)} \\
			&&&& {A_{\cal F}(x_2, c_2)} \\
			&&&& {f^{\ast}l_J(c_2)}
			\arrow["f"', from=3-3, to=4-3]
			\arrow["{s_{x_1}}", from=1-5, to=2-5]
			\arrow["u", from=3-3, to=1-5]
			\arrow["{f^{\ast}l_J(a_1)}"{pos=0.6}, from=4-3, to=2-5]
			\arrow["{f_i}", from=3-1, to=4-1]
			\arrow["{g_i}", from=3-1, to=3-3]
			\arrow["{f^{\ast}l_J(w_i)}", from=4-1, to=4-3]
			\arrow["{s_{x_i}}", from=3-4, to=4-4]
			\arrow["{f^{\ast}l_J(b_i^1)}"'{pos=0.6}, from=4-4, to=2-5]
			\arrow["{f^{\ast}l_J(z_i)}"', bend right=18, from=4-1, to=4-4]
			\arrow["{\gamma_i}"{pos=0.2}, bend left=-18, from=3-1, to=3-4]
			\arrow["{A_{\cal F}(p_i, b_i^1)}", from=3-4, to=1-5]
			\arrow["{s_{x_{2}}}", from=5-5, to=6-5]
			\arrow["{A_{\cal F}(q_i, b^2_i)}"{pos=0.7}, from=3-4, to=5-5]
			\arrow["{f^{\ast}l_J(b^2_i)}"', from=4-4, to=6-5]
			\arrow["v"{pos=0.7}, from=3-3, to=5-5]
			\arrow["{f^{\ast}l_J(a_2)}"'{pos=0.3}, from=4-3, to=6-5]
		\end{tikzcd}\]

		\item For any pair of objects and arrows in $\mathcal{G}(\mathbb D)$: $(k,u), (k',u'): (x_1,c_1) \to (x_2,c_2)$, and any object of the canonical relative site of $f$, $a: F \to f^*l_J(c)$, together with a morphism:  $h: c \to c_1$, and $v: F \to A_{\cal F}(x_1,c_1)$ such that $s_{x_1} \circ v = f^*l_J(h) \circ a$, there exist a set $I$ with an epimorphic family $g_i: F_i \to F$, objects $c_i$ of $\cal C$ and arrows $a_i: F_i \to f^*l_J(c_i)$ and $g_i': c_i \to c$ such that $f^*l_J(g_i') \circ a_i = a \circ g_i$, and objects of the comma in the image of $A$: $s_{x_i}: A_{\cal F}(x_i,c_i') \to f^*l_J(c_i')$ and arrows  $h_i: c_i \to c_i'$ and $v_i: F_i \to A_{\cal F}(x_i,c_i')$ such that $s_{x_i} \circ v_i = a_i \circ f^*l_J(h_i)$ with finally a morphism in $(e_i,u_i)$ in $\mathcal{G}(\mathbb D)$ such that $h_i \circ e_i = h \circ g_i'$ and, $v \circ g_i = A_{\cal F}(e_i,u_i') \circ v_i$ as we can visualize in the diagram:

        
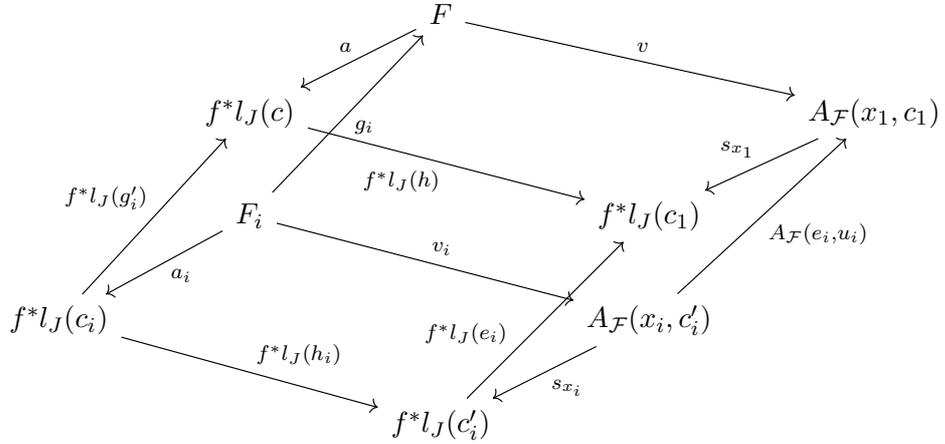
\[\begin{tikzcd}
	&& F \\
	& {f^*l_J(c)} &&& {A_{\cal F}(x_1,c_1)} \\
	& {F_i} && {f^*l_J(c_1)} \\
	{f^*l_J(c_i)} &&& {A_{\cal F}(x_i,c_i')} \\
	&& {f^*l_J(c_i')}
	\arrow["v", from=1-3, to=2-5]
	\arrow["{f^*l_J(h)}"', from=2-2, to=3-4]
	\arrow["a"', from=1-3, to=2-2]
	\arrow["{s_{x_1}}"', from=2-5, to=3-4]
	\arrow["{f^*l_J(g_i')}", from=4-1, to=2-2]
	\arrow["{g_i}"', from=3-2, to=1-3]
	\arrow["{f^*l_J(e_i)}"{pos=0.3}, from=5-3, to=3-4]
	\arrow["{A_{\cal F}(e_i,u_i)}"', from=4-4, to=2-5]
	\arrow["{v_i}", from=3-2, to=4-4]
	\arrow["{f^*l_J(h_i)}", from=4-1, to=5-3]
	\arrow["{a_i}", from=3-2, to=4-1]
	\arrow["{s_{x_i}}", from=4-4, to=5-3]
        \end{tikzcd}\]

        and such that $k \circ e_i = k' \circ e_i$ and that $\mathbb D(e_i)(u)\circ u_i = \mathbb D(e_i)(u') \circ u_i $. 
		
	\end{enumerate} 
	
	If $K$ is a Grothendieck topology on ${\cal G}({\mathbb D})$ contanining $J_{\mathbb D}$ then a flat relative to $\Sh({\cal C}, J)$ functor $A$ is $K$-continuous if and only if $A$ is moreover cover-preserving as a functor $({\cal G}({\mathbb D}), K)\to ((1_{\cal F} \downarrow f^{\ast} l_J), J_f)$. 
\end{prop}\qed	

To prove the next theorem, we will need the following result. Let us denote $p_{\mathbb D}^{K}$ for $p_{\mathbb D}$ regarded as a comorphism of sites $({\cal G}({\mathbb D}), K)\to ({\cal C}, J)$:

\begin{prop}\label{fibrationeta}
    
Let $p^K_{\mathbb D}: (\mathcal{G}({\mathbb D}),K) \to ({\cal C},J)$ be a relative site. Then the functor $\eta_{\mathcal{G}(\mathbb D)}: (\mathcal{G}(\mathbb D),K) \to ((1_{\Sh(\mathcal{G}(\mathbb D),K)}\downarrow C_{p^K_{\mathbb D}}^*l_J),J_{C_{p^K_{\mathbb D}}})$ is a cover-preserving morphism of fibrations. 
\end{prop}

\begin{proof}
The cover-preservation is immediate, and the fact that it is a morphism of fibrations follows from the following diagrams (for any $(x',c')$,  $(\overline{x},\overline{c})$ and $g: c \to c'$) being pullbacks: 
\[\begin{tikzcd}	{\mathcal{G}(\mathbb D)((\overline{x},\overline{c}),(\mathbb D(g)(x'),c)} & {{\cal C}(\overline{c},c)} \\	{\mathcal{G}(\mathbb D)((\overline{x},\overline{c}),(x',c'))} & {{\cal C}(\overline{c},c')}	\arrow["{(1,g)\circ -}", from=1-1, to=2-1]	\arrow["{p_{\mathbb D}}(-)"', from=1-1, to=1-2]	\arrow["{p_{\mathbb D}}(-)", from=2-1, to=2-2]	\arrow["{g \circ -}"', from=1-2, to=2-2] \end{tikzcd}\]
We just need to sheafify (respecting finite limits) and we obtain that the image of a cartesian arrow is cartesian.
\end{proof}

\begin{thm}\label{diaconescufibration}
	Let $f:{\cal F}\to \Sh({\cal C}, J)$ be a $\Sh({\cal C}, J)$-topos, ${\mathbb D}$ a $\cal C$-indexed category. Then we have an equivalence
	\[
	\textup{\bf Geom}_{\Sh({\cal C}, J)}([f], [C_{p_{\mathbb D}}])\simeq \textup{\bf Flat}_{\Sh({\cal C}, J)}({\cal G}({\mathbb D}), (1_{\cal F} \downarrow f^{\ast}l_J))
	\]
	between the category of geometric morphisms, relative to $\Sh({\cal C}, J)$, from $f$ to $C_{p_{\mathbb D}}$, and the category of flat relative to $\Sh({\cal C}, J)$ functors from ${\cal G}({\mathbb D})$ to $(1_{\cal F} \downarrow f^{\ast}l_J)$.
	
	If $K$ is a Grothendieck topology on ${\cal G}({\mathbb D})$ containing the Giraud topology $J_{\mathbb D}$ then the above equivalence restricts to an equivalence  
		\[
	\textup{\bf Geom}_{\Sh({\cal C}, J)}([f], [C_{p^{K}_{\mathbb D}}])\simeq \textup{\bf Flat}^{K}_{\Sh({\cal C}, J)}({\cal G}({\mathbb D}), (1_{\cal F} \downarrow f^{\ast}l_J)),
	\]
	where $\textup{\bf Flat}^{K}_{\Sh({\cal C}, J)}({\cal G}({\mathbb D}), (1_{\cal F}\downarrow f^{\ast}l_J))$ is the full subcategory of the category $\textup{\bf Flat}_{\Sh({\cal C}, J)}({\cal G}({\mathbb D}), (1_{\cal F}\downarrow f^{\ast}l_J))$ on the $K$-continuous functors. 
\end{thm}

\begin{proof}

In one direction, given a flat functor relative to $\Sh({\cal C},J)$, that is, a morphism of fibrations $A: {\cal G}({\mathbb D}) \to (1_{\cal F}\downarrow f^{\ast}l_J)$ which is also a morphism of ordinary sites, we know from from Proposition \ref{propmorphismfibrationsrelmorphismsites} that it induces a relative geometric morphism over $\Sh({\cal C},J)$, namely $\Sh(A\pi_{\cal F})$. 

For the other direction, let us exhibit a pseudo-inverse to this construction. Let $a: [f] \to [C_p] $ be a relative geometric morphism, by pre-composition of its inverse image with $l_K$, we get: $ a^*l_K: \mathcal{G}(\mathbb D) \to \cal F$. Then, we take its $\eta$-extension: $ \widetilde{a^*l_K}: (1_{\Sh(\mathcal{G}(\mathbb D),J_{\mathbb D})} \downarrow C_{p_{\mathbb D}}^*l_J) \to (1_{\cal F} \downarrow f^*l_J) $, which is a morphism of fibrations in the light of Proposition \ref{etaextensionproperties}. Finally, we precompose it with $\eta_{\mathcal{G}(\mathbb D)}$ to get a morphism of fibrations (by Proposition\ref{fibrationeta}): $\widetilde{a^*l_K} \circ \eta_{\mathcal{G}(\mathbb D)}: \mathcal{G}(\mathbb D) \to (1_{\cal F} \downarrow f^*l_J)$. This is in fact a morphism of relative sites: we see that $\pi_{\Sh({\cal C},J)} \circ \widetilde{a^*l_K} \circ \eta_{\mathcal{G}(\mathbb D)} = a^*l_K$ is a morphism of sites, but $\pi_{\Sh({\cal C},J)}$ is a dense morphism of sites, and $a^*l_K$ is a morphism of sites, thus $\widetilde{a^*l_K} \circ \eta_{\mathcal{G}(\mathbb D)}$ is also a morphism of sites. Moreover, using again $\pi_{\Sh({\cal C},J)} \circ \widetilde{a^*l_K} \circ \eta_{\mathcal{G}(\mathbb D)} = a^*l_K$, we obtain $\Sh(\pi_{\Sh({\cal C},J)})^* \circ \Sh(\widetilde{a^*l_K})^* \circ \Sh(\eta_{\mathcal{G}(\mathbb D)})^* = \Sh(a^*l_K)^*$, that is: $\widetilde{a^*l_K} \circ \eta_{\mathcal{G}(\mathbb D)} \simeq a^*l_K$, since $\pi_{\Sh({\cal C},J)}$ and $\eta_{\mathcal{G}(\mathbb D)}$ induce equivalences at the topos level. Hence, $\widetilde{a^*l_K} \circ \eta_{\mathcal{G}(\mathbb D)}$ is a flat functor relative to $({\cal C},J)$, and we also have the first part of the equivalence. 

To prove the other part of the equivalence, given a flat functor $A: \mathcal{G}(\mathbb D) \to (1_{\cal F} \downarrow f^*l_J)$ relative to $({\cal C},J)$, by Proposition \ref{etaextensionproperties} (iv), the following square commutes:

\[\begin{tikzcd}
	{(1_{\Sh(\mathcal{G}(\mathbb D),J_{\mathbb D})} \downarrow C_{p_{\mathbb D}}^*l_J)} & {(1_{\cal F} \downarrow f^*l_J)} \\
	{\mathcal{G}(\mathbb D)} & {(1_{\cal F} \downarrow f^*l_J)}
	\arrow["{\eta_{\mathcal{G}(\mathbb D)}}", from=2-1, to=1-1]
	\arrow["\simeq"', from=2-2, to=1-2]
	\arrow["{\widetilde{A}}", from=1-1, to=1-2]
	\arrow["A"', from=2-1, to=2-2]
\end{tikzcd}\]
Now, since $A$ and $\pi_{\cal F} \circ A$ induce the same geometric morphism ($\pi_{\cal F}$ being a dense morphism of sites), we have that $\widetilde{A} \simeq \widetilde{\pi_{\cal F} \circ A}$. Thus, we can simplify:  $\widetilde{\Sh(\pi_{\cal F}A)^*l_K} \circ \eta_{\mathcal{G}(\mathbb D)} = \widetilde{\pi_{\cal F}A} \circ \eta_{\mathcal{G}(\mathbb D)}$, and finally $\widetilde{\Sh(\pi_{\cal F}A)^*l_K} \circ \eta_{\mathcal{G}(\mathbb D)} = \widetilde{A} \circ \eta_{\mathcal{G}(\mathbb D)} = A$.

The adaptation of the proof to any topology containing the Giraud's one is immediate.

\end{proof}

\subsection{Approach via denseness conditions}

This approach stems from the construction of the canonical site of a \emph{relative geometric morphism}, that is, of a geometric morphism over a given base topos. This construction is a natural variant of the canonical relative site of a geometric morphism which takes into account the fact that the two domain and codomain morphisms are defined over the same base topos. 

\subsubsection{The canonical site of a relative geometric morphism}

Given a commutative diagram of geometric morphisms
\[\begin{tikzcd}
	{{\cal F}} && {{\cal F}'} \\
	\\
	& {{}\cal E}
	\arrow["{f'}", from=1-3, to=3-2]
	\arrow["f"', from=1-1, to=3-2]
	\arrow["g", from=1-1, to=1-3]
\end{tikzcd}\]
we regard $g$ as a relative geometric morphism $[f]\to [f']$ over $\cal E$. 

The inverse image $g^{\ast}$ of $g$ yields a morphism of relative sites:
\[
{g^{\ast}}_{\textup{fib}}:((1_{{\cal F}'}\downarrow f'^{\ast}), J_{f'}) \to ((1_{\cal F}\downarrow f^{\ast}), J_f)
\]

Let us denote by $\pi_{\cal E}$ (resp. $\pi'_{\cal E}$) the canonical projection functor from the comma category $(1_{{\cal F}}\downarrow f^{\ast}) \to {\cal E}$ (resp. $(1_{{\cal F}'}\downarrow f'^{\ast}) \to {\cal E}$). By the abstract duality of \cite{denseness} section 3, we can turn the morphism of sites $g^*_{fib}$ into a comorphism of sites inducing the same geometric morphism. More precisely, we have the (non commutative) diagram:

\[\begin{tikzcd}
	& {((1_{\cal F'} \downarrow f'^*) \downarrow g^*_{fib})} \\
	{(1_{\cal F'} \downarrow f'^*)} && {(1_{\cal F} \downarrow f^*)} \\
	& {\cal E}
	\arrow["\pi"', from=1-2, to=2-1]
	\arrow["{\pi_{\cal E}}"', from=2-1, to=3-2]
	\arrow["{\pi_g}", from=1-2, to=2-3]
	\arrow["{\pi_{\cal E}'}", from=2-3, to=3-2]
\end{tikzcd}\]

where $\pi_{\cal E}$ is the canonical comorphism of sites associated to the canonical relative site of the geometric morphism $f$, same for $\pi_{\cal E}'$ with $f'$, $\pi$ a dense comorphism of sites (inducing an equivalence), and $\pi_g$ the comorphism of sites inducing $g$. The situation may seem more complex, but the point is that we are now reduced to only study one type of morphisms: the comorphisms. This square gives us back the wished triangle at the topos level, but it doesn't commute at the site level. To study whether or not a triangle is commutative at the topos level, we want to see if we can reduce this square at the site level to a commutative one. The commutativity at the topos level will follow from the functoriality of the construction of geometric morphisms from comorphisms of sites. This restriction to a commutative square is given by:

\begin{defn}
The \emph{diagonal category} of $g:[f]\to [f']$ is the full subcategory $\Delta_{\cal E}(g)$ of $(1_{(1_{\cal F}\downarrow f^{\ast})}\downarrow g^{\ast}_{fib})$ on the objects $(\chi, \xi, \chi\to {g^{\ast}}_{\textup{fib}}(\xi))$ such that $\pi_{\cal E}(\chi)=\pi'_{\cal E}(\xi)$ and the $\cal E$-component of the arrow $\chi\to g^{\ast}(\xi)$ is the identity.  

More concretely, the category $\Delta_{\cal E}(g)$ can be described as a category having as objects the quintuples 
\[
(F, E, F', \delta:F\to f'^{\ast}(E), u:F\to g^{\ast}(F'))
\]
and as arrows from
\[
(F_1, E_1, F'_1, \delta_1:F_1\to f'^{\ast}(E_1), u_1:F_1\to g^{\ast}(F'_1))
\]
 to 
\[ 
(F_2, E_2, F'_2, \delta_2:F_2\to f'^{\ast}(E_2), u_2:F_2\to g^{\ast}(F'_2))
\]
the triplets of arrows 
\[
(\alpha:F_1 \to F_2, \beta:E_1\to E_2, \gamma:F'_1 \to F'_2)
\] 
satisfying the obvious commutativity conditions.

\end{defn}

\begin{remark}
Note that this category is indexed by $\cal E$: an object over $E$ is the data of two arrows $ F \to f^*E $ and $F' \to f'^*E$ together with an arrow between the first and the image of the second by $g^*_{fib}$. This fibration acts as the usual fibrations of comma categories, but in double: the pullback of the two arrows is well-behaved in the light of the  preservation  of finite limits by $g^*$, and the arrow induced between these two pullbacks is obtained by the usual functoriality of the pullback. As $(1_{\cal F} \downarrow f^*)$ and $(1_{\cal F'} \downarrow f'^*)$ are stacks, and $g^*$ preserves covering families and finite limits, this fibration is in fact also a stack. Indeed if we have compatible local data of two such arrows together with another between them, we can glue the two arrows into two global arrows because of the stack properties of $(1_{\cal F} \downarrow f^*)$ and $(1_{\cal F'} \downarrow f'^*)$. The global arrow between these two global arrows is then given by the prestack property of $(1_{\cal F} \downarrow f^*)$ (sheaf property of the hom-functors): the local data of arrows between the localizations of these two global arrows is equivalent to the global data of an arrow between these two global arrows. Also, the fact that hom-functors of this fibration are sheaves is inherited of the fact that those of $(1_{\cal F} \downarrow f^*)$, $(1_{\cal F'} \downarrow f'^*)$ have this property (being stacks, and in particular prestacks) and the finite-limite cover preservation of $g^*$.
\end{remark}

Recall that the appropriate topology on the category $(1_{(1_{\cal F}\downarrow f^{\ast})}\downarrow g^{\ast})$ is the one induced by the canonical projection to $(1_{\cal F}\downarrow f^{\ast})$, equivalently, by the canonical projection to $\cal F$. We denote it, abusing notation (since this notation has already been used above), by $J_f$.

The interest of this construction is that, unlike the classial relative site of $g$, this enriched version of it admits canonical projection functors to the base topos $\cal E$, as well as to each of the canonical relative sites of the geometric morphisms $f$ and $f'$, as displayed in the following diagram:

\[\begin{tikzcd}
	&& \Delta_{\cal E}(g) && {(1_{(1_{\cal F}\downarrow f^{\ast})}\downarrow g^{\ast})} \\
	{(1_{{\cal F}'}\downarrow f'^{\ast})} && {(1_{\cal F}\downarrow f^{\ast})} \\
	& {{\cal E}}
	\arrow["{g^{\ast}}", from=2-1, to=2-3]
	\arrow[hook, from=1-3, to=1-5]
	\arrow[from=1-5, to=2-3]
	\arrow["{\pi'_{\cal E}}"', from=2-1, to=3-2]
	\arrow["{\pi_{\cal E}}", from=2-3, to=3-2]
	\arrow["{\pi^{f}_{g}}", from=1-3, to=2-3]
	\arrow["{\pi^{f'}_{g}}"', from=1-3, to=2-1]
\end{tikzcd}\]

We now have:
\[
\pi_{\cal E}\circ \pi_{g}^{f}=\pi'_{\cal E}\circ \pi_{g}^{f'},
\]
while we do not have such a commutativity condition for the whole category $((1_{\cal F}\downarrow f^{\ast})\downarrow g^{\ast})$. The following proposition explains that if the triangle at the topos level commutes, then the diagonal subcategory is dense, \textit{i.e.}, in some way, our restricted square carries all the wished information of the non-commutative one:

\begin{prop}\label{prop:relativesiteofrelativemorphism}
	Let $g:[f]\to [f']$ be a relative geometric morphism. Then the diagonal category	
	$\Delta_{\cal E}(g)$ of $g$ is a $J_f$-dense subcategory of $(1_{(1_{\cal F}\downarrow f^{\ast})}\downarrow g^{\ast})$, whence $\pi^{f'}_{g}$ yields a comorphism of sites $(\Delta_{\cal E}(g), J_f|_{\Delta_{\cal E}(g)}) \to (1_{{\cal F}'}\downarrow f'^{\ast}, J_{f'})$ over $\cal E$ such that $g\cong C_{\pi^{f'}_{g}}$:
\end{prop}

\begin{proof}
We have to show that every object of the category $(1_{(1_{\cal F}\downarrow f^{\ast})}\downarrow g^{\ast})$ can be $J_f$-covered by objects lying in $\Delta_{\cal E}(g)$. 

An object of $(1_{(1_{\cal F}\downarrow f^{\ast})}\downarrow g^{\ast})$ is a triplet consisting in an object $(F, E, \alpha:F\to f^{\ast}(E))$ of the category $(1_{\cal F}\downarrow f^{\ast})$, an object $(F', E', \alpha':F'\to f'^{\ast}(E'))$ of the category $(1_{{\cal F}'}\downarrow f'^{\ast})$ and an arrow 
\[
(u, e):(F, E, \alpha:F\to f^{\ast}(E)) \to {g^{\ast}}_{\textup{fib}}(F', E', \alpha':F'\to f'^{\ast}(E'))
\]
in the category $(1_{\cal F}\downarrow f^{\ast})$.

Given such a triplet, consider the triplet consisting of the object $(F, E\times E', <\alpha, f^{\ast}(e)\circ \alpha'>:F \to f^{\ast}(E\times E')\cong f^{\ast}(E)\times f^{\ast}(E'))$ of $(1_{\cal F}\downarrow f^{\ast})$, the object $$(f'^{\ast}(E)\times F', E\times E', 1_{f'^{\ast}(E)}\times \alpha':f'^{\ast}(E)\times F' \to f^{\ast}(E)\times f^{\ast}(E')\cong f^{\ast}(E\times E'))$$ of the category $(1_{{\cal F}'}\downarrow f'^{\ast})$ and the arrow from the former object to the image under ${g^{\ast}}_{\textup{fib}}$ of the latter given by the pair $(<\alpha, u>, 1_{E\times E'})$ (modulo the isomorphism $f^{\ast}(E)\cong g^{\ast}(f'^{\ast}(E))$).

Now, the second triplet admits an arrow in the category $(1_{(1_{\cal F}\downarrow f^{\ast})}\downarrow g^{\ast})$ to the first triplet whose $\cal F$-component of the first component is the identity on $F$; so this arrow is $J_f$-covering by definition of the topology $J_f$. Since the objects $(F, E\times E', <\alpha, f^{\ast}(e)\circ \alpha'>:F \to f^{\ast}(E\times E')\cong f^{\ast}(E)\times f^{\ast}(E'))$ and in the second triplet  $$(f'^{\ast}(E)\times F', E\times E', 1_{f'^{\ast}(E)}\times \alpha':f'^{\ast}(E)\times F' \to f^{\ast}(E)\times f^{\ast}(E')\cong f^{\ast}(E\times E'))$$  have the same $\cal E$-component, namely $E\times E'$, the second triplet, as an object of the category $(1_{(1_{\cal F}\downarrow f^{\ast})}\downarrow g^{\ast})$, actually lies in $\Delta_{\cal E}(g)$. So the denseness condition is satisfied, as required.   
\end{proof}

The following result is relevant in connection of morphisms between localic geometric morphisms.

\begin{prop}\label{prop:preservationmono}
	Let $g:[f:{\cal F}\to {\cal E}]\to [f':{\cal F}' \to {\cal E}]$ be a relative geometric morphism. Then the functor
	\[
	{g^{\ast}}_{\textup{fib}}:((1_{{\cal F}'}\downarrow f'^{\ast}), J_{f'}) \to ((1_{\cal F}\downarrow f^{\ast}), J_f)
	\]
	restricts to a morphism of sites
	\[
	((1_{{\cal F}'}\downarrow^{\textup{Sub}} f'^{\ast}), J_{f'}|_{(1_{{\cal F}'}\downarrow^{\textup{Sub}} f'^{\ast})}), ((1_{\cal F}\downarrow^{\textup{Sub}} f^{\ast}) \to J_f |_{((1_{\cal F}\downarrow^{\textup{Sub}} f^{\ast})})
	\]
\end{prop}

\begin{proof}
	It suffices to observe that an object $(F', E', \alpha:F'\to f'^{\ast}(E'))$ can be regarded as an arrow $(\alpha, 1):(f'^{\ast}(E'), E', 1_{f'^{\ast}(E')})$ in the category $(1_{{\cal F}'}\downarrow f'^{\ast})$ (and similarly for the relative category of the morphism $f$). Now, if $\alpha$ is monic then the corresponding arrow $(\alpha, 1)$ if also monic, and hence is sent by the functor ${g^{\ast}}_{\textup{fib}}$ to a monic arrow towards the identity, which in turn identifies with an object of the subcategory $(1_{\cal F} \downarrow^{\textup{Sub}} f^{\ast})$, as required. 
\end{proof}

\subsubsection{The diagonal site of a relative morphism of sites}

We have seen in Proposition \ref{prop:relativesiteofrelativemorphism} that if a morphism of sites
$$m:((1_{{\cal F}'}\downarrow f'^{\ast}), J_{f'}) \to ((1_{\cal F}\downarrow f^{\ast}), J_f)$$
is of the form ${g^{\ast}}_{\textup{fib}}$ for a relative geometric morphism $g:[f]\to [f']$ over $\cal E$ then the diagonal subcategory of $(1_{(1_{\cal F}\downarrow f^{\ast})} \downarrow m)$, is dense in it. Conversely, if $m$ is a morphism of sites with this property then it necessarily induces a geometric morphism defined over $\cal E$, as, denoting by $\pi_{1}$ and $\pi_{2}$ are the two projections from the diagonal category of $m$ respectively to $(1_{\cal F'} \downarrow f'^{\ast})$ and $(1_{\cal F}\downarrow f^{\ast})$, $\pi_{2}$ thus induces an equivalence of toposes under which $\Sh(m)$ corresponds to $C_{\pi_{1}}$, and the following commutativity condition for comorphisms of sites holds: $\pi'_{\cal E}\circ \pi_{1}=\pi_{\cal E}\circ \pi_{2}$.
 
Recall that in \ref{morphsitesinducing} we had defined for every fiber $c$ a functor $A_c$. The index $c$ was fixed, but we can bring all of them together and define a functor \[
A_{\cal C}: (p'\downarrow 1_{\cal C}) \to (p\downarrow 1_{\cal C})
\]
sending an object $(d, c, u:p(d)\to c)$ of $(p\downarrow 1_{\cal C})$ to the object $(A(d), c, u\phi_d: p'A(d))\to c)$ of $(p'\downarrow 1_{\cal C})$. This functor is defined in the spirit of $g^*_{fib}$, but at the level of sites.

Indeed, with the same arguments as for the $A_c$, this functor makes the following diagram commute:

\[\begin{tikzcd}
	{(p'\downarrow 1_{\cal C})} && {(p\downarrow 1_{\cal C})} \\
	{((1_{\textup{\bf Sh}({\cal D'}, K)} \downarrow C_{p'}^{\ast}), J_{C_{p'}})} && {((1_{\textup{\bf Sh}({\cal D}, K)} \downarrow C_{p}^{\ast}), J_{C_{p}})}
	\arrow["{\xi_{p'}}"', from=1-1, to=2-1]
	\arrow["{\xi_{p}}", from=1-3, to=2-3]
	\arrow["{A_{\cal C}}", from=1-1, to=1-3]
	\arrow["{\textbf{Sh}(A)^{\ast}}", from=2-1, to=2-3]
\end{tikzcd}\]
 All this discussion provides the inspiration for the following general criterion for a morphism between relative sites to induce a relative geometric morphism:    

\begin{thm}\label{thm:charaterizationrelativemorphismdenseness}
	Let $({\cal C}, J)$ be a small-generated site, $p:({\cal D}, K) \to ({\cal C}, J)$ and $p':({\cal D}', K') \to ({\cal C}, J)$ comorphisms of sites and $A:({\cal D}, K) \to ({\cal D'}, K')$ a morphism of sites together with a natural transformation $\phi: p'\circ A \Rightarrow p$. Then the following conditions are equivalent:
	\begin{enumerate}[(i)]
		\item The pair $(A,\phi)$ is a morphism of sites over $\cal C$, that is, $\Sh(A)$ is a relative geometric morphism $[C_{p'}]\to [C_{p}]$ over $\Sh({\cal C}, J)$;
		
		\item The diagonal subcategory $(1_{(p' \downarrow {\cal C})}\downarrow^{\ast} A_{\cal C})$ of $(1_{(p' \downarrow {\cal C})}\downarrow A_{\cal C})$ constituted of the triplets $(d',d,c,c', pd \xrightarrow{u} c, p'd \xrightarrow{u'} c', d' \xrightarrow{g} Ad, c' \xrightarrow{f} c) $ such that $c'=c$ is $\widetilde{K_{A_{\cal C}}}$-dense in it, with $\widetilde{K_{A_{\cal C}}}$ having for covering sieves those sent by the two consecutive first projections to covering ones in $\cal D'$:
		\[\begin{tikzcd}
			&& {(1_{(p' \downarrow {\cal C})}\downarrow^{\ast} A_{\cal C})} & {(1_{(p' \downarrow {\cal C})}\downarrow A_{\cal C})} \\
			{(p \downarrow {\cal C})} && {(p' \downarrow {\cal C})} \\
			& {{\cal C}}
			\arrow["{A_{\cal C}}", from=2-1, to=2-3]
			\arrow[hook, from=1-3, to=1-4]
			\arrow[from=1-4, to=2-3]
			\arrow["{\pi^1_{\cal C}}"', from=2-1, to=3-2]
			\arrow["{\pi^2_{\cal C}}", from=2-3, to=3-2]
			\arrow["{\pi_{2}}", from=1-3, to=2-3]
			\arrow["{\pi_{1}}"', from=1-3, to=2-1]
		\end{tikzcd}\]

	\end{enumerate}
\end{thm} 

\begin{proof}

 To show that (i) $\Rightarrow$ (ii) we can use the characterization of $A$ being a morphism of relative sites in term of the $A_c$ being filtering. Let us take any object of the non diagonal category, it can be summarized as:

\[\begin{tikzcd}
	{p'(d')} && {p'(Ad)} \\
	\\
	{c'} && c
	\arrow["{p'(g)}", from=1-1, to=1-3]
	\arrow["{u'}"', from=1-1, to=3-1]
	\arrow["u\phi_d", from=1-3, to=3-3]
	\arrow["f"', from=3-1, to=3-3]
\end{tikzcd}\]

In order to cover it with objects from the diagonal category, we may want to bring together the components of this object over the same index. By the first property of the relative filteredness, we can have locally over $c'$, for some covering family $(f_i)_{i\in I}$:

\[\begin{tikzcd}
	& {p'(d')} && {p'(Ad)} \\
	\\
	& {c'} && c \\
	&& {c'} \\
	{p'(d'_i)} &&&& {p'(Ad_i)}
	\arrow["{p'(g)}", from=1-2, to=1-4]
	\arrow["{u'}", from=1-2, to=3-2]
	\arrow["u\phi_d"', from=1-4, to=3-4]
	\arrow["f"', from=3-2, to=3-4]
	\arrow["1", from=4-3, to=3-2]
	\arrow["f"', from=4-3, to=3-4]
	\arrow["{u'p'(f_i)}", from=5-1, to=4-3]
	\arrow["{p'(g_i)}", from=5-1, to=5-5]
	\arrow["{p'(f_i)}", from=5-1, to=1-2]
	\arrow["{u_i\phi_{d_i}}", from=5-5, to=4-3]
\end{tikzcd}\]

We can now, by the second property of the relative filteredness, locally connect $p'(Ad)$ and the $p'(Ad_i)$:

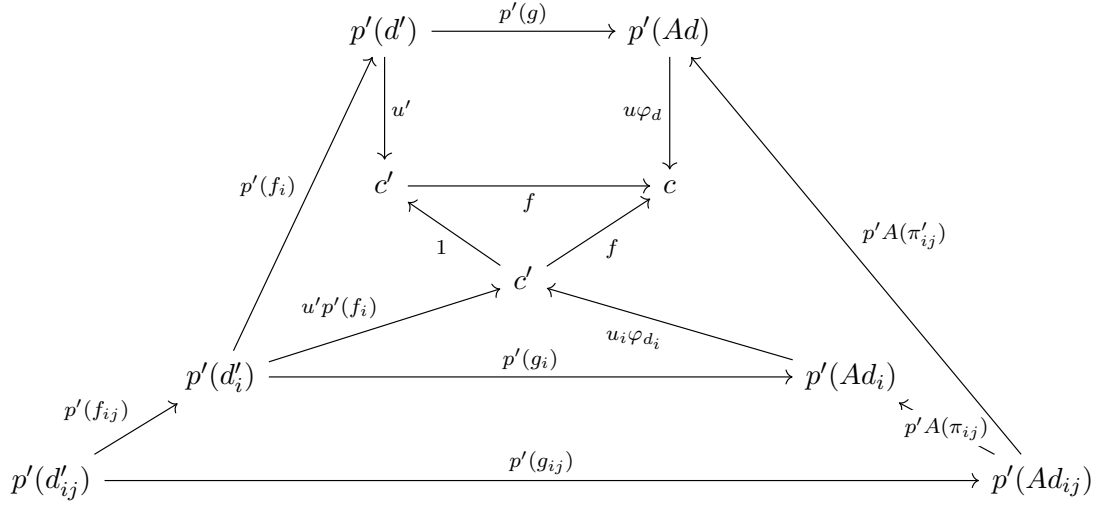
\[\begin{tikzcd}
	&& {p'(d')} && {p'(Ad)} \\
	\\
	&& {c'} && c \\
	&&& {c'} \\
	& {p'(d'_i)} &&&& {p'(Ad_i)} \\
	{p'(d'_{ij })} &&&&&& {p'(Ad_{ij})}
	\arrow["{p'(g)}", from=1-3, to=1-5]
	\arrow["{u'}", from=1-3, to=3-3]
	\arrow["u\phi_d"', from=1-5, to=3-5]
	\arrow["f"', from=3-3, to=3-5]
	\arrow["1", from=4-4, to=3-3]
	\arrow["f"', from=4-4, to=3-5]
	\arrow["{u'p'(f_i)}", from=5-2, to=4-4]
	\arrow["{p'(g_i)}", from=5-2, to=5-6]
	\arrow["{p'(f_i)}", from=5-2, to=1-3]
	\arrow["{u_i\phi_{d_i}}", from=5-6, to=4-4]
	\arrow["{p'(f_{ij})}", from=6-1, to=5-2]
	\arrow["{p'A(\pi_{ij})}"{description}, from=6-7, to=5-6]
	\arrow["{p'(g_{ij})}", from=6-1, to=6-7]
	\arrow["{p'A(\pi'_{ij})}"', from=6-7, to=1-5]
\end{tikzcd}\]

The multicomposition of the $f_i$ and $f_{ij}$ is covering by transitivity, so the obvious arrows and objects depicted on the diagram are covering for the topology on the non diagonal category (the covering are exactly those which have their left component covering), and coming from the diagonal category.

To see that (ii) $\Rightarrow$ (i), we recall that, in the diagram of the point (ii) in the proposition, all the projections are comorphisms of sites (cf denseness conditions). Moreover, $\pi_1$ induces $\Sh (A)$, $\pi_2$ induces an equivalence since the upper inclusion does by hypothesis, as the arrow coming from the left hand comma category (cf 3.4.2 \cite{denseness}), and $\pi_{\cal C}^1$ and $\pi_{\cal C}^2$ induce the structure geometric morphisms over $\Sh ({\cal C}, J)$. Since we restricted to the the diagonal category, the square of projections commutes, and all these projections are comorphisms of sites, inducing the wishing triangle at the topos level, therefore the triangle at the topos level commutes.
\end{proof}

Here is an explicit description of this denseness condition:

\begin{prop}
	Under the assumption of the above proposition:
	\begin{enumerate}[(i)]
		\item The category $(1_{(p' \downarrow {\cal C})}\downarrow^{\ast} A_{\cal C})$ is isomorphic to the one having as objects the $(d',d,c, d \xrightarrow{v} Ad, pd \xrightarrow{u} c ) $ with $d'$ object of $\cal D'$, $d$ of $\cal D$, and $c$ of $\cal C$.
		
		\item The $\widetilde{K_{A_{\cal C}}}$-denseness condition for $(1_{(p \downarrow {\cal C})}\downarrow^{\ast} A_{\cal C})\hookrightarrow (1_{(p \downarrow {\cal C})}\downarrow A_{\cal C})$ can be explicitly reformulated as follows: 

for every object of the comma category $(1_{(p \downarrow {\cal C})}\downarrow A_{\cal C})$
\begin{center}
    
\[\begin{tikzcd}
	{p'(d')} && {p'(Ad)} \\
	\\
	{c'} && c
	\arrow["{p'(g)}", from=1-1, to=1-3]
	\arrow["{u'}"', from=1-1, to=3-1]
	\arrow["u\phi_d", from=1-3, to=3-3]
	\arrow["f"', from=3-1, to=3-3]
\end{tikzcd}\]
\end{center}

there exist objects $c_i$ of $\cal C$, objects $d_i$ of $\cal D$, objects $d'_i$ of $\cal D'$, with arrows $p'(d'_i) \xrightarrow{u'_i} c_i$ and $p(d_i) \xrightarrow{u_i} c_i $, and an arrow $d'_i \xrightarrow{g_i} A(d_i)$ giving diagonal objects of the comma category

\[\begin{tikzcd}
	{p'(d'_i)} && {p'(Ad_i)} \\
	& {c_i}
	\arrow["{u'_i}"', from=1-1, to=2-2]
	\arrow["{u_i\phi_{d_i}}", from=1-3, to=2-2]
	\arrow["{p'(g_i)}", from=1-1, to=1-3]
\end{tikzcd}\]

and arrows $c_i \xrightarrow{w'_i} c'$ and $c_i \xrightarrow{w_i} c$, arrows $d'_i \xrightarrow{v'_i} d'$ and $d_i \xrightarrow{v_i} d$ giving morphisms between these diagonal objects and the first one, that is, the following diagram commutes:

\[\begin{tikzcd}
	{p'(d'_i)} &&&& {p'(Ad_i)} \\
	&& {c_i} \\
	& {c'} && c \\
	\\
	& {p'(d)} && {p'(Ad)}
	\arrow["{u'_i}"', from=1-1, to=2-3]
	\arrow["{u_i\phi_{d_i}}", from=1-5, to=2-3]
	\arrow["{p'(g_i)}", from=1-1, to=1-5]
	\arrow["{w'_i}", from=2-3, to=3-2]
	\arrow["{w_i}"', from=2-3, to=3-4]
	\arrow["f"', from=3-2, to=3-4]
	\arrow["{u'}"', from=5-2, to=3-2]
	\arrow["{u\phi_d}", from=5-4, to=3-4]
	\arrow["{p'(g)}", from=5-2, to=5-4]
	\arrow["{p'(v'_i)}"', from=1-1, to=5-2]
	\arrow["{p'(v_i)}", from=1-5, to=5-4]
\end{tikzcd}\]

with the $v'_i$ making a covering sieve.

\end{enumerate}

\begin{proof}
The first point is immediately deduced from the fact that in the commutative triangle representing an object of the diagonal category, it is not necessary to mention one of the three arrows, it being the composite of the two others. The second point states that the square on the bottom of the diagram, an object of $(1_{(p' \downarrow {\cal C})}\downarrow A_{\cal C})$, can be covered by objects of the diagonal category, being pictured as the triangle on the top of of the diagram, with the first component of the morphisms between them, the $v_i'$, being covering (a sieve being covering in this category exactly when its projection on the first component is covering in $\cal D'$). 
\end{proof}

\end{prop}

\bibliography{biblio}{}
\bibliographystyle{abbrv}

\vspace{1cm}

\textsc{Léo Bartoli} 

\vspace{0.2cm}
{\small \textsc{Department of Mathematics, ETH Zurich, Rämistrasse 101
8092 Zurich, Switzerland.}\\
\emph{E-mail address:} \texttt{lbartoli@ethz.ch}

\vspace{0.2cm}

{\small \textsc{Istituto Grothendieck ETS, Corso Statuto 24, 12084 Mondovì, Italy.}\\
	\emph{E-mail address:} \texttt{leo.bartoli@ctta.igrothendieck.org}}

\vspace{0.6cm}

\textsc{Olivia Caramello} 

\vspace{0.2cm}
{\small \textsc{Dipartimento di Scienza e Alta Tecnologia, Universit\`a degli Studi dell'Insubria, via Valleggio 11, 22100 Como, Italy.}\\
	\emph{E-mail address:} \texttt{olivia.caramello@uninsubria.it}}

\vspace{0.2cm}

{\small \textsc{Istituto Grothendieck ETS, Corso Statuto 24, 12084 Mondovì, Italy.}\\
	\emph{E-mail address:} \texttt{olivia.caramello@igrothendieck.org}}

\end{document}